\documentclass[12pt,english]{article}
\pdfoutput=1
\usepackage[T1]{fontenc}
\usepackage[latin9]{inputenc}
\usepackage{geometry}
\usepackage{mathtools, amsmath, amsthm, amssymb}
\usepackage{stmaryrd}

\makeatletter
\numberwithin{equation}{section}
\numberwithin{figure}{section}
\theoremstyle{plain}
\newtheorem{thm}{\protect\theoremname}
\theoremstyle{plain}
\newtheorem{lem}[thm]{\protect\lemmaname}
\theoremstyle{plain}
\newtheorem{prop}[thm]{\protect\propositionname}
\theoremstyle{plain}
\newtheorem{conjecture}[thm]{\protect\conjecturename}
\theoremstyle{plain}
\newtheorem{question}[thm]{\protect\questionname}
\theoremstyle{definition}
\newtheorem{problem}[thm]{\protect\problemname}

\usepackage{float}
\usepackage{allrunes}
\usepackage{ytableau}
\usepackage{bm}
\usepackage{hyperref}
\usepackage{xparse}
\usepackage{microtype}
\usepackage{shuffle}
\usepackage{tikz}
\usetikzlibrary{arrows, arrows.meta, quotes, positioning, shapes, decorations.pathmorphing, decorations.pathreplacing}

\DeclareMathOperator{\Des}{Des}
\DeclareMathOperator{\Pk}{Pk}
\DeclareMathOperator{\pk}{pk}
\DeclareMathOperator{\Lpk}{Lpk}
\DeclareMathOperator{\lpk}{lpk}
\DeclareMathOperator{\Rpk}{Rpk}
\DeclareMathOperator{\rpk}{rpk}
\DeclareMathOperator{\Epk}{Epk}
\DeclareMathOperator{\epk}{epk}
\DeclareMathOperator{\Comp}{Comp}

\DeclareMathOperator{\des}{des}
\DeclareMathOperator{\st}{st}
\DeclareMathOperator{\cst}{cst}
\DeclareMathOperator{\std}{std}

\DeclareMathOperator{\Val}{Val}
\DeclareMathOperator{\val}{val}
\DeclareMathOperator{\maj}{maj}

\DeclareMathOperator{\QSym}{QSym}
\DeclareMathOperator{\cQSym}{cQSym}
\DeclareMathOperator{\Span}{span}

\DeclareMathOperator{\cyc}{cyc}

\usepackage{titlesec}
\titleformat{\section}
  {\normalfont\large\bfseries}{\thesection.}{.8ex}{} 
\titleformat{\subsection}
  {\normalfont\bfseries}{\thesubsection.}{.8ex}{} 

\let\originalleft\left
\let\originalright\right
\renewcommand{\left}{\mathopen{}\mathclose\bgroup\originalleft}
\renewcommand{\right}{\aftergroup\egroup\originalright}

\newcommand{\leqnomode}{\tagsleft@true\let\veqno\@@leqno}
\newcommand{\reqnomode}{\tagsleft@false\let\veqno\@@eqno}
\reqnomode

\makeatother

\usepackage{babel}
\providecommand{\conjecturename}{Conjecture}
\providecommand{\lemmaname}{Lemma}
\providecommand{\problemname}{Problem}
\providecommand{\propositionname}{Proposition}
\providecommand{\questionname}{Question}
\providecommand{\theoremname}{Theorem}

\usepackage{authblk}
\title{On kernels of descent statistics}
\author[1]{William L.\ Clark\thanks{\tt{s82wclar@uni-bonn.de}}}
\author[2]{Yan Zhuang\thanks{\tt{yazhuang@davidson.edu}}}
\affil[1]{Mathematical Institute, University of Bonn} 
\affil[2]{Department of Mathematics and Computer Science, Davidson College}

\begin{document}
\maketitle
\begin{abstract}
The kernel $\mathcal{K}^{\st}$ of a descent statistic $\st$, introduced by Grinberg, is a subspace of the algebra $\QSym$ of quasisymmetric functions defined in terms of $\st$-equivalent compositions, and is an ideal of $\QSym$ if and only if $\st$ is shuffle-compatible. This paper continues the study of kernels of descent statistics, with emphasis on the peak set $\Pk$ and the peak number $\pk$. The kernel $\mathcal{K}^{\Pk}$ in particular is precisely the kernel of the canonical projection from $\QSym$ to Stembridge's algebra of peak quasisymmetric functions, and is the orthogonal complement of Nyman's peak algebra. We prove necessary and sufficient conditions for obtaining spanning sets and linear bases for the kernel $\mathcal{K}^{\st}$ of any descent statistic $\st$ in terms of fundamental quasisymmetric functions, and give characterizations of $\mathcal{K}^{\Pk}$ and $\mathcal{K}^{\pk}$ in terms of the fundamental basis and the monomial basis of $\QSym$. Our results imply that the peak set and peak number statistics are $M$-binomial, confirming a conjecture of Grinberg.
\end{abstract}
\textbf{\small{}Keywords:}{\small{} quasisymmetric functions, compositions, permutation statistics, shuffle-compatibility, descents, peaks}
{\let\thefootnote\relax\footnotetext{YZ was partially supported by an AMS-Simons Travel Grant and NSF grant DMS-2316181.}}
{\let\thefootnote\relax}

\section{Introduction}

\setcounter{thm}{-1}

This paper studies ideals of the ring of quasisymmetric functions associated with shuffle-compatible permutation statistics. We begin by giving the relevant definitions.

We call $\pi$ a \textit{permutation} of \textit{length} $n$ if it is a sequence of $n$ distinct positive integers, displayed as the word $\pi=\pi_{1}\pi_{2}\cdots\pi_{n}$. Let $\left|\pi\right|$ denote the length of a permutation $\pi$ and let $\mathfrak{P}_{n}$ denote the set of all permutations of length $n$. Note that $\mathfrak{P}_{n}$ contains the set $\mathfrak{S}_{n}$ of permutations of $[n]\coloneqq\{1,2,\dots,n\}$, as every permutation in $\mathfrak{S}_{n}$ can be written in one-line notation as a sequence of $n$ distinct positive integers, but $\mathfrak{P}_{n}$ and $\mathfrak{S}_{n}$ are not the same. For example, $83416$ is an element of $\mathfrak{P}_{5}$ but not of $\mathfrak{S}_{5}$.

Given a permutation $\pi\in\mathfrak{P}_{n}$, define the \textit{standardization} $\std\pi$ of $\pi$ to be the unique permutation in $\mathfrak{S}_{n}$ obtained by replacing the smallest letter of $\pi$ by 1, the second smallest by 2, and so on. For example, $\std 83416=52314$. A \textit{permutation statistic} is a function $\st$ defined on permutations such that $\st\pi=\st\sigma$ whenever $\std\pi=\std\sigma$.\footnote{The condition that $\std\pi=\std\sigma$ implies $\st\pi=\st\sigma$ is not actually used in this paper, but it plays a role in the theory of shuffle-compatible permutation statistics.} Note that every permutation statistic $\st$ defined on $\mathfrak{S}_{n}$ can be extended to $\mathfrak{P}_{n}$ by taking $\st\pi\coloneqq\st(\std\pi)$.

A classical example of a permutation statistic is the descent set, defined as follows. We say that $i\in[n-1]$ is a \textit{descent} of a permutation $\pi\in\mathfrak{P}_{n}$ if $\pi_{i}>\pi_{i+1}$, and the \textit{descent set}
\[
\Des\pi\coloneqq\{\,i\in[n-1]:\pi_{i}>\pi_{i+1}\,\}
\]
of $\pi$ is its set of descents. 

The information contained inside the descent set can also be encoded as an integer composition. First observe that every permutation can be uniquely decomposed into a sequence of maximal increasing consecutive subsequences called \textit{increasing runs}. The \textit{descent composition} of $\pi$, denoted $\Comp\pi$, is the composition whose parts are the lengths of the increasing runs of $\pi$ in the order that they appear. For example, the increasing runs of $\pi=379426$ are $379$, $4$, and $26$, so $\Comp\pi=(3,1,2)$. If $\Comp\pi=(j_{1},j_{2},\dots,j_{m})$, then the descent set of $\pi$ is given by 
\begin{equation}
\Des\pi=\{j_{1},j_{1}+j_{2},\dots,j_{1}+j_{2}+\cdots+j_{m-1}\}.\label{e-Des}
\end{equation}
Conversely, if $\Des\pi=\{i_{1}<i_{2}<\dots<i_{m}\}$, then 
\begin{equation}
\Comp\pi=(i_{1},i_{2}-i_{1},\dots,i_{m}-i_{m-1},n-i_{m})\label{e-Comp}
\end{equation}
where $n$ is the length of $\pi$. We shall use the notations $L\vDash n$ and $\left|L\right|=n$ to indicate that $L$ is a composition of $n$, and $\mathcal{C}$ for the set of all compositions. By convention, we allow the ``empty composition'' $\emptyset$ to be a composition of 0; this is the descent composition of the empty permutation, the sole element of $\mathfrak{P}_0$.

A permutation statistic $\st$ is called a \textit{descent statistic} if $\Comp\pi=\Comp\sigma$ implies $\st\pi=\st\sigma$\textemdash that is, if $\st$ depends only on the descent composition, or equivalently, on the descent set and the length. Whenever $\st$ is a descent statistic, we may write $\st L$ for the value of $\st$ on any permutation with descent composition $L$. Besides the descent set $\Des$, examples of descent statistics include the descent number $\des$, major index $\maj$, peak set $\Pk$, peak number $\pk$, exterior peak set $\Epk$, valley set $\Val$, and valley number $\val$; their definitions will be given later.

For a descent statistic $\st$, two compositions $J$ and $K$ are said to be $\st$-\textit{equivalent} if $\st J=\st K$ and $\left|J\right|=\left|K\right|$; when this is the case, we write $J\sim_{\st}K$. Then the \textit{kernel} of $\st$, denoted $\mathcal{K}^{\st}$, is the subspace
\[
\mathcal{K}^{\st}\coloneqq\Span\{\,F_{J}-F_{K}:J\sim_{\st}K\,\}
\]
of the $\mathbb{Q}$-algebra $\QSym$ of quasisymmetric functions, where $F_{L}$ refers to the \textit{fundamental quasisymmetric function}
\[
F_{L}\coloneqq\sum_{\substack{i_{1}\leq i_{2}\leq\cdots\leq i_{n}\\
i_{j}<i_{j+1}\text{ if }j\in\Des L
}
}x_{i_{1}}x_{i_{2}}\cdots x_{i_{n}}
\]
where $n=\left|L\right|$.

The kernel of a descent statistic was defined by Grinberg in \cite{Grinberg2018,Grinberg2018a},\footnote{The paper \cite{Grinberg2018a} is an extended version of \cite{Grinberg2018}, containing additional results and more detailed proofs. We will cite the published version \cite{Grinberg2018} unless referring to content which only appears in the extended version \cite{Grinberg2018a}.} which was a continuation of the work by Gessel and Zhuang on shuffle-compatible permutation statistics~\cite{Gessel2018}. Gessel and Zhuang defined the shuffle algebra $\mathcal{A}^{\st}$ of a shuffle-compatible permutation statistic (see Section \ref{ss-shuf} for definitions), and showed that whenever a descent statistic $\st$ is shuffle-compatible, its shuffle algebra is isomorphic to a quotient of $\QSym$. The kernel $\mathcal{K}^{\st}$ is an ideal of $\QSym$ if and only if $\st$ is shuffle-compatible, and is in fact the kernel of the canonical projection from $\QSym$ to $\mathcal{A}^{\st}$.

Using their framework, Gessel and Zhuang proved that a number of descent statistics are shuffle-compatible and gave explicit descriptions for their shuffle algebras, but left the shuffle-compatibility of the exterior peak set $\Epk$ as a conjecture. Grinberg proved the shuffle-compatibility of $\Epk$ and gave two characterizations of the ideal $\mathcal{K}^{\Epk}$: one in terms of the fundamental quasisymmetric functions and another in terms of the \textit{monomial quasisymmetric functions}
\[
M_{L}\coloneqq\sum_{i_{1}<i_{2}<\cdots<i_{m}}x_{i_{1}}^{j_{1}}x_{i_{2}}^{j_{2}}\cdots x_{i_{m}}^{j_{m}}
\]
where $L=(j_{1},j_{2},\dots,j_{m})$. Both characterizations have nice combinatorial descriptions involving the underlying compositions. We state these two results together in the following theorem.

\begin{thm}[{Grinberg \cite[Propositions 103 and 105]{Grinberg2018}}] \label{t-Epk}
Given compositions $J=(j_{1},j_{2},\dots,j_{m})$ and $K$, we write $J\rightarrow K$ if there exists $l\in\{2,3,\dots,m\}$ for which $j_{l}>2$ and 
\[
K=(j_{1},\dots,j_{l-1},1,j_{l}-1,j_{l+1},\dots,j_{m}),
\]
and we write $J\triangleright K$ if there exists $l\in\{2,3,\dots,m\}$ for which $j_{l}>2$ and 
\[
K=(j_{1},\dots,j_{l-1},2,j_{l}-2,j_{l+1},,\dots,j_{m}).
\]
Then the ideal $\mathcal{K}^{\Epk}$ is spanned \textup{(}as a $\mathbb{Q}$-vector space\textup{)} in the following ways\textup{:}
\[
\mathcal{K}^{\Epk}=\Span\{\,F_{J}-F_{K}:J\rightarrow K\,\}=\Span\{\,M_{J}+M_{K}:J\triangleright K\,\}.
\]
\end{thm}

Grinberg suggests a systematic study of kernels of descent statistics. One direction of research concerns the $M$-binomial property: a descent statistic $\st$ is said to be $M$-\textit{binomial} if $\mathcal{K}^{\st}$ can be spanned by elements of the form $\lambda M_{J}+\mu M_{K}$ with $\lambda,\mu\in\mathbb{Q}$. For example, it follows from Theorem \ref{t-Epk} that $\Epk$ is $M$-binomial, and Grinberg gave a list of other descent statistics\textemdash including the peak set $\Pk$ and the peak number $\pk$\textemdash for which computational evidence suggests are $M$-binomial as well \cite[Question 107]{Grinberg2018}.

\subsection{Overview of results}

The purpose of the present work is to continue the study of kernels of descent statistics initiated by Grinberg, with emphasis on the kernels of the peak set and the peak number. Both $\Pk$ and $\pk$ are shuffle-compatible statistics, so ${\cal K}^{\Pk}$ and ${\cal K}^{\pk}$ are ideals of $\QSym$. We will give characterizations\textemdash analogous to the ones given by Grinberg for $\mathcal{K}^{\Epk}$\textemdash for both ideals, and the ones given in terms of the monomial quasisymmetric functions show that $\Pk$ and $\pk$ are indeed $M$-binomial. 

To obtain characterizations for the kernels in terms of the fundamental quasisymmetric functions, we prove necessary and sufficient conditions for a subset of $\{\,F_{J}-F_{K}:J\sim_{\st}K\,\}$ to be: (a) a spanning set of $\mathcal{K}^{\st}$, and (b) linearly independent. Note that every subset of $\{\,F_{J}-F_{K}:J\sim_{\st}K\,\}$ can be written as 
\[
\mathcal{F}_{S}^{\st}\coloneqq\{\,F_{J}-F_{K}:(J,K)\in S\,\}
\]
for some subset $S$ of $\{\,(J,K):J\sim_{\st}K\,\}$, and we associate to $S$ a directed graph $G_{S}$ with vertex set $\mathcal{C}$ and edge set $S$\textemdash i.e., there is an edge from $J$ to $K$ if and only if $(J,K)\in S$.

For us, a \textit{connected component} of a directed graph refers to a connected component of the underlying undirected graph, and a directed graph is called a \textit{forest} if its underlying undirected graph has no cycles.

\begin{thm} \label{t-Fsst}
Let $\st$ be a descent statistic and let $S\subseteq\{\,(J,K):J\sim_{\st}K\,\}$.
\begin{enumerate}
\item [\normalfont{(a)}] The kernel $\mathcal{K}^{\st}$ is spanned by $\mathcal{F}_{S}^{\st}$ if and only if the connected components of $G_{S}$ are precisely the $\st$-equivalence classes of $\mathcal{C}$\textemdash i.e., $J\sim_{\st}K$ if and only if $J$ and $K$ are in the same connected component of $G_{S}$.
\item [\normalfont{(b)}] The set $\mathcal{F}_{S}^{\st}$ is linearly independent if and only if $G_{S}$ is a forest.
\end{enumerate}
\end{thm}
We will apply Theorem \ref{t-Fsst} to the peak set and peak number statistics. Given compositions $J=(j_{1},j_{2},\dots,j_{m})$ and $K$, we write: 
\begin{itemize}
\item $J\rightarrow_{1}K$ if there exists $l\in[m]$ for which $j_{l}>2$
and 
\[
K=(j_{1},\dots,j_{l-1},1,j_{l}-1,j_{l+1},\dots,j_{m});
\]
\item $J\rightarrow_{2}K$ if $j_{m}=2$ and 
\[
K=(j_{1},\dots,j_{m-1},1,1);
\]
\item $J\rightarrow_{3}K$ if $j_{i}\leq2$ for all $i\in[m]$, $j_{m}=1$,
$j_{l}=1$ and $j_{l+1}=2$ for some $l\in[m-2]$, and 
\[
K=(j_{1},\dots,j_{l-1},j_{l+1},j_{l},j_{l+2},\dots,j_{m}).
\]
\end{itemize}

See Figure \ref{f-arrows} for all relations $\rightarrow_{1}$, $\rightarrow_{2}$, and $\rightarrow_{3}$ among compositions of at most 5. Ignoring edge labels, this is also the subgraph of $G_{S}$ for $S=\{\,(J,K):J\rightarrow_{1}K,\,J\rightarrow_{2}K,\text{ or }J\rightarrow_{3}K\,\}$ induced by the compositions of at most 5. 

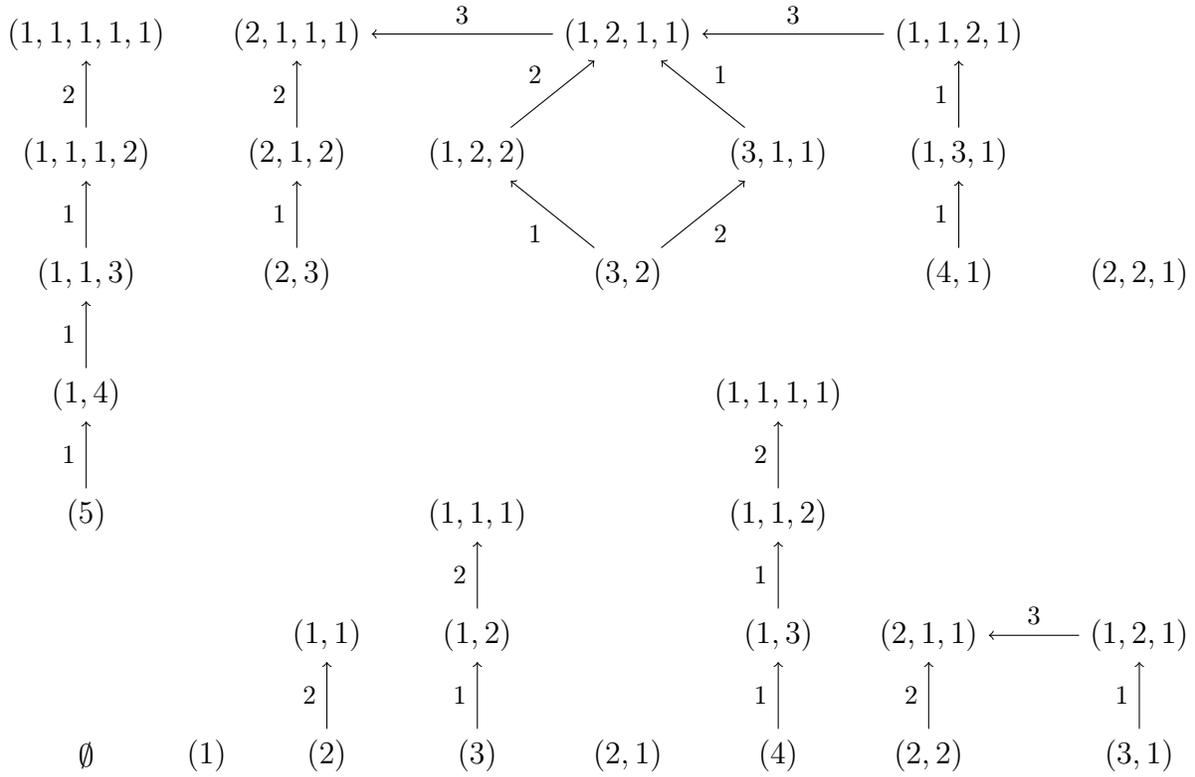
\begin{figure}
\noindent \begin{centering}
\begin{center}
\begin{tikzpicture}[scale=0.4,auto,every edge quotes/.style = {font=\footnotesize}]

\node (1) at (0,0) {$\emptyset$};

\node (2) at (4,0) {$(1)$};

\node (3) at (8,0) {$(2)$};
\node (4) at (8,4) {$(1,1)$};
\draw[-to] (3) edge["2"] (4);

\node (5) at (13,0) {$(3)$};
\node (6) at (13,4) {$(1,2)$};
\node (7) at (13,8) {$(1,1,1)$};
\draw[-to] (5) edge["1"] (6);
\draw[-to] (6) edge["2"] (7);

\node (7) at (18,0) {$(2,1)$};

\node (8) at (23,0) {$(4)$};
\node (9) at (23,4) {$(1,3)$};
\node (10) at (23,8) {$(1,1,2)$};
\node (11) at (23,12) {$(1,1,1,1)$};
\draw[-to] (8) edge["1"] (9);
\draw[-to] (9) edge["1"] (10);
\draw[-to] (10) edge["2"] (11);

\node (12) at (28,0) {$(2,2)$};
\node (13) at (28,4) {$(2,1,1)$};
\draw[-to] (12) edge["2"] (13);

\node (14) at (35,0) {$(3,1)$};
\node (15) at (35,4) {$(1,2,1)$};
\draw[-to] (14) edge["1"] (15);

\draw[-to] (15) edge["3",above] (13);

\node (16) at (0,8) {$(5)$};
\node (17) at (0,12) {$(1,4)$};
\node (18) at (0,16) {$(1,1,3)$};
\node (19) at (0,20) {$(1,1,1,2)$};
\node (20) at (0,24) {$(1,1,1,1,1)$};
\draw[-to] (16) edge["1"] (17);
\draw[-to] (17) edge["1"] (18);
\draw[-to] (18) edge["1"] (19);
\draw[-to] (19) edge["2"] (20);

\node (21) at (7,16) {$(2,3)$};
\node (22) at (7,20) {$(2,1,2)$};
\node (23) at (7,24) {$(2,1,1,1)$};
\draw[-to] (21) edge["1"] (22);
\draw[-to] (22) edge["2"] (23);

\node (24) at (18,16) {$(3,2)$};
\node (25) at (13,20) {$(1,2,2)$};
\node (26) at (23,20) {$(3,1,1)$};
\node (27) at (18,24) {$(1,2,1,1)$};
\draw[-to] (24) edge["1",below left] (25);
\draw[-to] (24) edge["2",below right] (26);
\draw[-to] (25) edge["2",above left] (27);
\draw[-to] (26) edge["1",above right] (27);

\node (28) at (29,16) {$(4,1)$};
\node (29) at (29,20) {$(1,3,1)$};
\node (30) at (29,24) {$(1,1,2,1)$};
\draw[-to] (28) edge["1"] (29);
\draw[-to] (29) edge["1"] (30);

\draw[-to] (27) edge["3",above] (23);
\draw[-to] (30) edge["3",above] (27);

\node (31) at (35,16) {$(2,2,1)$};

\end{tikzpicture}
\end{center}
\par\end{centering}
\vspace{-10bp}
\caption{\label{f-arrows}All relations $\rightarrow_{1}$, $\rightarrow_{2}$, and $\rightarrow_{3}$ among compositions of at most 5.}
\bigskip{}
\end{figure}

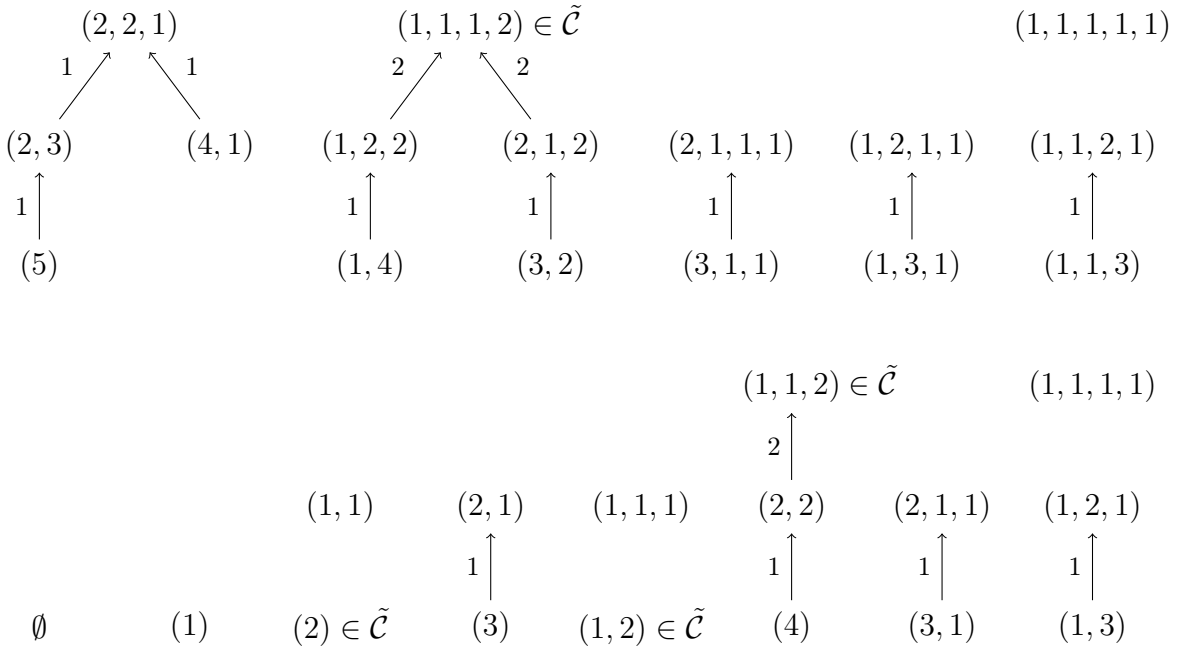
\begin{figure}
\noindent \begin{centering}
\begin{center}
\begin{tikzpicture}[scale=0.4,auto,every edge quotes/.style = {font=\footnotesize}]

\node (1) at (0,0) {$\emptyset$};

\node (2) at (5,0) {$(1)$};

\node (3) at (10,0) {$(2)\in \tilde{\mathcal{C}}$};

\node (4) at (10,4) {$(1,1)$};

\node (5) at (15,0) {$(3)$};
\node (6) at (15,4) {$(2,1)$};
\draw[-to] (5) edge["1"] (6);

\node (7) at (20,0) {$(1,2)\in \tilde{\mathcal{C}}$};

\node (8) at (20,4) {$(1,1,1)$};

\node (9) at (25,0) {$(4)$};
\node (10) at (25,4) {$(2,2)$};
\node (11) at (25,8) {$(1,1,2)$};
\node (11a) at (27.7,8.2) {$\in \tilde{\mathcal{C}}$};
\draw[-to] (9) edge["1"] (10);
\draw[-to] (10) edge["2"] (11);

\node (12) at (30,0) {$(3,1)$};
\node (13) at (30,4) {$(2,1,1)$};
\draw[-to] (12) edge["1"] (13);

\node (14) at (35,0) {$(1,3)$};
\node (15) at (35,4) {$(1,2,1)$};
\draw[-to] (14) edge["1"] (15);

\node (16) at (35,8) {$(1,1,1,1)$};

\node (17) at (0,12) {$(5)$};
\node (18) at (0,16) {$(2,3)$};
\node (19) at (6,16) {$(4,1)$};
\node (20) at (3,20) {$(2,2,1)$};
\draw[-to] (17) edge["1"] (18);
\draw[-to] (18) edge["1",above left] (20);
\draw[-to] (19) edge["1",above right] (20);

\node (21) at (11,12) {$(1,4)$};
\node (22) at (11,16) {$(1,2,2)$};
\node (23) at (17,12) {$(3,2)$};
\node (24) at (17,16) {$(2,1,2)$};
\node (25) at (14,20) {$(1,1,1,2)$};
\node (25a) at (17.2,20.2) {$\in \tilde{\mathcal{C}}$};
\draw[-to] (21) edge["1"] (22);
\draw[-to] (22) edge["2",above left] (25);
\draw[-to] (23) edge["1"] (24);
\draw[-to] (24) edge["2",above right] (25);

\node (26) at (23,12) {$(3,1,1)$};
\node (27) at (23,16) {$(2,1,1,1)$};
\draw[-to] (26) edge["1"] (27);

\node (28) at (29,12) {$(1,3,1)$};
\node (29) at (29,16) {$(1,2,1,1)$};
\draw[-to] (28) edge["1"] (29);

\node (30) at (35,12) {$(1,1,3)$};
\node (31) at (35,16) {$(1,1,2,1)$};
\draw[-to] (30) edge["1"] (31);

\node (32) at (35,20) {$(1,1,1,1,1)$};

\end{tikzpicture}
\end{center}
\par\end{centering}
\vspace{-10bp}
\caption{\label{f-arrowsM}All relations $\triangleright_{1}$ and $\triangleright_{2}$, and elements of $\tilde{\mathcal{C}}$, among compositions of at most 5.}
\end{figure}

\pagebreak

\begin{thm} \label{t-F}
The ideals $\mathcal{K}^{\Pk}$ and $\mathcal{K}^{\pk}$ are spanned \textup{(}as $\mathbb{Q}$-vector spaces\textup{)} in the following ways\textup{:}
\begin{enumerate}
\item [\normalfont{(a)}]$\mathcal{K}^{\Pk}=\Span\{\,F_{J}-F_{K}:J\rightarrow_{1}K\text{ or }J\rightarrow_{2}K\,\}$
\item [\normalfont{(b)}]$\mathcal{K}^{\pk}=\Span\{\,F_{J}-F_{K}:J\rightarrow_{1}K,\ J\rightarrow_{2}K,\text{ or }J\rightarrow_{3}K\,\}$
\end{enumerate}
\end{thm}

In fact, there is a simple way to trim the spanning sets in Theorem \ref{t-F} to obtain bases for $\mathcal{K}^{\Pk}$ and $\mathcal{K}^{\pk}$ which still admit nice descriptions, but we will delay the statement of these results until later.

We then use Theorem \ref{t-F} and a change-of-basis argument to obtain characterizations for these ideals in terms of monomial quasisymmetric functions. Given compositions $J=(j_{1},j_{2},\dots,j_{m})$ and $K$, we write: 
\begin{itemize}
\item $J\triangleright_{1}K$ if there exists $l\in[m]$ for which $j_{l}>2$ and 
\[
K=(j_{1},\dots,j_{l-1},2,j_{l}-2,j_{l+1},\dots,j_{m});
\]
\item $J\triangleright_{2}K$ if $j_{m}=2$, $j_{l}=2$ for some $l\in[m-1]$, and
\[
K=(j_{1},\dots,j_{l-1},1,1,j_{l+1},\dots,j_{m-1},2);
\]
\item $J\in\mathcal{\tilde{C}}$ if $J=(1^{m-1},2)=(\underset{m-1\text{ parts}}{\underbrace{1,\dots,1}},2)$. (Note that $(2)\in\mathcal{\tilde{C}}$ in the case $m=1$.)
\end{itemize}

See Figure \ref{f-arrowsM} for all relations $\triangleright_{1}$ and $\triangleright_{2}$, and elements of $\mathcal{\tilde{C}}$, among compositions of at most 5.

\begin{thm} \label{t-M}
The ideals $\mathcal{K}^{\Pk}$ and $\mathcal{K}^{\pk}$ are spanned \textup{(}as $\mathbb{Q}$-vector spaces\textup{)} in the following ways\textup{:}
\begin{enumerate}
\item [\normalfont{(a)}]$\mathcal{K}^{\Pk}=\Span\left(\{\,M_{J}+M_{K}:J\triangleright_{1}K\text{ or }J\triangleright_{2}K\,\}\cup\{\,M_{J}:J\in\mathcal{\tilde{C}}\,\}\right)$
\item [\normalfont{(b)}]$\mathcal{K}^{\pk}=\Span\!\Big(\{\,M_{J}+M_{K}:J\triangleright_{1}K\text{ or }J\triangleright_{2}K\,\}\cup\{\,M_{J}:J\in\mathcal{\tilde{C}}\,\}\cup\{\,M_{J}-M_{K}:J\rightarrow_{3}K\,\}\Big)$
\end{enumerate}
Therefore, $\Pk$ and $\pk$ are $M$-binomial.
\end{thm}

We note that the shuffle algebra $\mathcal{A}^{\Pk}$ is isomorphic to Stembridge's algebra $\Pi$ of peak quasisymmetric functions \cite{Stembridge1997}, which occupies an important position in the theory of combinatorial Hopf algebras \cite{Aguiar2006}. Our results about $\mathcal{K}^{\Pk}$ can thus be translated into results about the canonical projection map from $\QSym$ to $\Pi$. Moreover, $\mathcal{K}^{\Pk}$ is the orthogonal complement of Nyman's peak algebra \cite{Nyman2003} (as a nonunital subalgebra of the noncommutative symmetric functions), and similarly $\mathcal{K}^{\pk}$ is the orthogonal complement of Schocker's Eulerian peak algebra \cite{Schocker2005}.

Peaks and valleys in permutations are related via complementation, and this symmetry will be exploited to obtain from Theorem \ref{t-F} an analogous result for the kernels of the valley set and valley number statistics. We will give the statement of this result later.

\subsection{Outline}

This paper is organized as follows. Section 2 focuses on background material, including definitions and properties of various descent statistics, the connection between quasisymmetric functions and shuffle-compatibility, and a result from linear algebra which will be used in our later proofs. In Section 3, we first prove a couple general results about spanning sets and linear independence in arbitrary vector spaces, which together imply Theorem \ref{t-Fsst}. We then use Theorem~\ref{t-Fsst} to produce our characterizations for the ideals $\mathcal{K}^{\Pk}$ and $\mathcal{K}^{\pk}$ given in Theorem~\ref{t-F}, and also obtain linear bases for $\mathcal{K}^{\Pk}$ and $\mathcal{K}^{\pk}$ by trimming the spanning sets from Theorem~\ref{t-F} in a simple way. Section 4 will be devoted to the proof of Theorem~\ref{t-M}, and Section 5 on the ideals $\mathcal{K}^{\Val}$ and $\mathcal{K}^{\val}$ for the valley set and valley number statistics. We end in Section 6 with a discussion of future directions of research.

All vector spaces, algebras, linear combinations, spans, and related notions are over the field $\mathbb{Q}$ except in Section \ref{ss-Fsst} (where we work over an arbitrary field of characteristic $\neq 2$).

\numberwithin{thm}{section}

\section{\label{s-prelim}Preliminaries}

\subsection{\label{ss-permstat}Descent statistics}

Recall that a permutation statistic $\st$ is a \textit{descent statistic} if $\st\pi=\st\sigma$ whenever $\Comp\pi=\Comp\sigma$. The descent set $\Des$ is clearly a descent statistic, and other examples of descent statistics include the following:

\begin{itemize}
\item The \textit{descent number} $\des$ and \textit{major index} $\maj$.
Given a permutation $\pi$, define 
\[
\des\pi\coloneqq\left|\Des\pi\right|\quad\text{and}\quad\maj\pi\coloneqq\sum_{k\in\Des\pi}k
\]
to be its number of descents and its sum of descents, respectively.

\item The \textit{peak set} $\Pk$ and the \textit{peak number} $\pk$. Given $\pi\in\mathfrak{P}_{n}$, we say that $i\in\{2,3,\dots,n-1\}$ is a\textit{ peak} of $\pi$ if $\pi_{i-1}<\pi_{i}>\pi_{i+1}$. Then $\Pk\pi$ is defined to be the set of peaks of $\pi$ and $\pk\pi$ its number of peaks.

\item The \textit{exterior peak set} $\Epk$ and the \textit{exterior peak number} $\epk$. Given $\pi\in\mathfrak{P}_{n}$, we say that $i\in[n]$ is an \textit{exterior peak} of $\pi$ if $i$ is a peak of $\pi$, if $i=1$ and $\pi_{1}>\pi_{2}$,  if $i=n$ and $\pi_{n-1}<\pi_{n}$, or if $i=n=1$. Then $\Epk\pi$ is defined to be the set of exterior peaks of $\pi$ and $\epk\pi$ its number of exterior peaks. 

\end{itemize}

For example, if $\pi=713649$, then $\des\pi=2$, $\maj\pi=5$, $\Pk\pi=\{4\}$, $\pk\pi=1$, $\Epk\pi=\{1,4,6\}$, and $\epk\pi=3$. Several additional descent statistics---associated with valleys, left peaks and right peaks---will be introduced later in this paper. Other descent statistics which we do not consider in this work include those based on double descents, alternating descents, and biruns; see \cite{Gessel2018,Zhuang2016} for definitions.

Next, recall that the notation $\Des L$ refers to the descent set of any permutation with descent composition $L$, and from Equation (\ref{e-Des}), we have 
\[
\Des L=\{j_{1},j_{1}+j_{2},\dots,j_{1}+j_{2}+\cdots+j_{m-1}\}
\]
for $L=(j_{1},j_{2},\dots,j_{m})$. Similarly, we can use Equation (\ref{e-Comp}) to define $\Comp$ on subsets: given $C=\{i_{1}<i_{2}<\cdots<i_{m}\}\subseteq[n-1]$, let 
\[
\Comp C\coloneqq(i_{1},i_{2}-i_{1},\dots,i_{m}-i_{m-1},n-i_{m}).
\]
Then $\Des$ and $\Comp$ are inverse bijections between compositions of $n$ and subsets of $[n-1]$.

For our characterizations of the ideals $\mathcal{K}^{\Pk}$ and $\mathcal{K}^{\pk}$, it will be helpful to have an explicit formula for the $\Pk$ and $\pk$ statistics on compositions. The lemma below follows immediately from the fact that the peaks of a permutation occur precisely at the end of its non-final increasing runs of length at least 2.

\begin{lem} \label{l-PkpkL}
Let $L=(j_{1},j_{2},\dots,j_{m})$ be a composition.
Then:
\begin{enumerate}
\item [\normalfont{(a)}]$\Pk L=\left\{ \,\sum_{i=1}^{k}j_{i}:j_{k}\geq2\text{ and }k\in [m-1]\,\right\}$
\item [\normalfont{(b)}]$\pk L=\left|\{\,k\in [m-1]:j_{k}\geq2\,\}\right|$
\end{enumerate}
\end{lem}

\vspace{5bp}

\subsection{\label{ss-qsym}Quasisymmetric functions and shuffle-compatibility}

Quasisymmetric functions arose in the early work of Stanley as generating functions for $P$-partitions \cite{Stanley1972}, were first defined and studied per se by Gessel \cite{Gessel1984}, and are now ubiquitous in algebraic combinatorics. We review some elementary definitions and results surrounding quasisymmetric functions, emphasizing their role in the theory of shuffle-compatibility; see \cite[Section 7.19]{Stanley2001}, \cite[Section 5]{Grinberg2020}, and \cite{Luoto2013} for further references.

Let $x_{1},x_{2},\dots$ be commuting variables. A formal power series $f\in\mathbb{Q}[[x_{1},x_{2},\dots]]$ of bounded degree is called a \textit{quasisymmetric function} if for any positive integers $a_{1},a_{2},\dots,a_{k}$, if $i_{1}<i_{2}<\cdots<i_{k}$ and $j_{1}<j_{2}<\cdots<j_{k}$ then
\[
[x_{i_{1}}^{a_{1}}x_{i_{2}}^{a_{2}}\cdots x_{i_{k}}^{a_{k}}]\,f=[x_{j_{1}}^{a_{1}}x_{j_{2}}^{a_{2}}\cdots x_{j_{k}}^{a_{k}}]\,f,
\]
i.e., the monomials $x_{i_{1}}^{a_{1}}x_{i_{2}}^{a_{2}}\cdots x_{i_{k}}^{a_{k}}$ and $x_{j_{1}}^{a_{1}}x_{j_{2}}^{a_{2}}\cdots x_{j_{k}}^{a_{k}}$ have the same coefficients in $f$. Let $\QSym_{n}$ denote the vector space of quasisymmetric functions homogeneous of degree $n$, and let
\[
\QSym\coloneqq\bigoplus_{n=0}^{\infty}\QSym_{n}.
\]

The monomial quasisymmetric functions $\{M_{L}\}_{L\vDash n}$ and the fundamental quasisymmetric functions $\{F_{L}\}_{L\vDash n}$ defined in the introduction are two bases of $\QSym_{n}$, and since (for $n\geq1$) there are $2^{n-1}$ compositions of $n$, it follows that $\QSym_{n}$ has dimension $2^{n-1}$. Through the inverse bijections $\Comp$ and $\Des$, we may also index the monomial and fundamental quasisymmetric functions by subsets $C$ of $[n-1]$ in writing 
\[
M_{n,C}\coloneqq M_{\Comp C}\quad\text{and}\quad F_{n,C}\coloneqq F_{\Comp C},
\]
and it will sometimes be convenient for us to do so.

Let us recall the change-of-basis formulas between the monomial and fundamental bases. First, we say that $J\vDash n$ \textit{refines} (or is a \textit{refinement} of) $K\vDash n$ if $\Des K\subseteq\Des J$. Informally, this amounts to saying that we can obtain $K$ from $J$ by combining some of its adjacent parts. For example, we have that $J=(2,1,3,1,1,2)$ refines $K=(3,5,2)$ because $\Des J=\{2,3,6,7,8\}$ contains $\Des K=\{3,8\}$, and indeed we have $2+1=3$ and $3+1+1=5$. Let us write $J\leq K$ if $J$ refines $K$. Then we have
\[
F_{L}=\sum_{K\leq L}M_{K},\quad\text{or equivalently,}\quad F_{n,C}=\sum_{C\subseteq B\subseteq[n-1]}M_{n,B}.
\]
Part (a) of the next lemma then follows from inclusion-exclusion. Parts (b)\textendash (c) appear as Propositions 5.10 (b)\textendash (c) of \cite{Grinberg2018a}.

\begin{lem} \label{l-MtoF}
Let $C\subseteq[n-1]$. 
\begin{enumerate}
\item [\normalfont{(a)}]We have
\[
M_{n,C}=\sum_{\substack{C\subseteq B\subseteq[n-1]}}(-1)^{\left|B\backslash C\right|}F_{n,B}.
\]
\item [\normalfont{(b)}]Suppose that $k\in[n-1]$ and $k\notin C$. Then
\[
M_{n,C}+M_{n,C\cup\{k\}}=\sum_{\substack{C\subseteq B\subseteq[n-1]\\k\notin B}}(-1)^{\left|B\backslash C\right|}F_{n,B}.
\]
\item [\normalfont{(c)}]Suppose that $k\in[n-1]$, $k\notin C$, and $k-1\notin C\cup\{0\}$.
Then 
\[
M_{n,C}+M_{n,C\cup\{k\}}=\sum_{\substack{C\subseteq B\subseteq[n-1]\\k,k-1\notin B}}(-1)^{\left|B\backslash C\right|}(F_{n,B}-F_{n,B\cup\{k-1\}}).
\]
\end{enumerate}
\end{lem}

Our present work on kernels of descent statistics only concerns the vector space structure of the quasisymmetric functions, but it is worthwhile to also discuss the ring structure of $\QSym$ as it will allow us to understand the background and motivation for studying these kernels. Let $\pi\in\mathfrak{P}_{m}$ and $\sigma\in\mathfrak{P}_{n}$ be \textit{disjoint} permutations\textemdash that is, they have no letters in common. Then we say that $\tau\in\mathfrak{P}_{m+n}$ is a \textit{shuffle} of $\pi$ and $\sigma$ if both $\pi$ and $\sigma$ are subsequences of $\tau$, and we let $\pi\shuffle\sigma$ denote the set of shuffles of $\pi$ and $\sigma$. For example, given $\pi=13$ and $\sigma=42$, we have 
\[
\pi\shuffle\sigma=\{1342,1432,1423,4132,4123,4213\}.
\]
The product of two fundamental quasisymmetric functions is given by
\begin{equation}
F_{L}F_{K}=\sum_{\tau\in\pi\shuffle\sigma}F_{\Comp\tau}\label{e-Fprod}
\end{equation}
where $\pi$ and $\sigma$ are any disjoint permutations satisfying $\Comp\pi=L$ and $\Comp\sigma=K$. It follows that $\QSym$ is a graded subalgebra of $\mathbb{Q}[[x_{1},x_{2},\dots]]$.

In order for the product formula (\ref{e-Fprod}) to make sense, the multiset $\{\,\Comp\tau:\tau\in\pi\shuffle\sigma\,\}$ must only depend on the descent compositions of $\pi$ and $\sigma$; equivalently, in terms of descent sets, the multiset $\{\,\Des\tau:\tau\in\pi\shuffle\sigma\,\}$ only depends on $\Des\pi$, $\Des\sigma$, and the lengths of $\pi$ and $\sigma$. More generally, we say that a permutation statistic $\st$ is called \textit{shuffle-compatible} if for any disjoint permutations $\pi$ and $\sigma$, the multiset $\{\,\st\tau:\tau\in\pi\shuffle\sigma\,\}$ giving the distribution of $\st$ over $\pi\shuffle\sigma$ depends only on $\st\pi$, $\st\sigma$, and the lengths of $\pi$ and $\sigma$. In other words, (\ref{e-Fprod}) implies that the descent set $\Des$ is a shuffle-compatible permutation statistic, which is implicit in Stanley's theory of $P$-partitions \cite{Stanley1972}; the shuffle-compatibility of the statistics $\des$, $\maj$, and $(\des,\maj)$ follow from Stanley's work as well.

Before proceeding, we note that $\QSym$ additionally has the structure of a dendriform algebra \cite{Grinberg2017a} and a Hopf algebra (and is in fact the terminal object in the category of combinatorial Hopf algebras \cite{Aguiar2006}). The dendriform structure of $\QSym$ is relevant to shuffle-compatibility and the kernels of descent statistics; this connection is not important for our present work but will be touched on in Section \ref{ss-future}.

\subsection{\label{ss-shuf}Shuffle algebras and the kernels \texorpdfstring{$\mathcal{K}^{\protect\st}$}{Kst}}

Motivated by the shuffle-compatibility results implicit in Stanley's work on $P$-partitions, Gessel and Zhuang~\cite{Gessel2018} formalized the notion of a shuffle-compatible permutation statistic and built a framework for investigating this phenomenon centered around the shuffle algebra of a shuffle-compatible statistic. Let us outline this construction below.

We say that permutations $\pi$ and $\sigma$ are $\st$-\textit{equivalent} if $\st\pi=\st\sigma$ and $\left|\pi\right|=\left|\sigma\right|$. We write the $\st$-equivalence class of $\pi$ as $[\pi]_{\st}$. For a shuffle-compatible statistic $\st$, we associate to $\st$ a $\mathbb{Q}$-algebra in the following way. First, associate to $\st$ a $\mathbb{Q}$-vector space by taking as a basis the $\st$-equivalence classes of permutations. We give this vector space a multiplication by taking 
\[
[\pi]_{\st}[\sigma]_{\st}=\sum_{\tau\in\pi\shuffle\sigma}[\tau]_{\st},
\]
which is well-defined if and only if $\st$ is shuffle-compatible. The resulting algebra ${\cal A}^{\st}$ is called the \textit{shuffle algebra} of $\st$. Observe that ${\cal A}^{\st}$ is graded by length\textemdash i.e., $[\pi]_{\st}$ belongs to the $n$th graded component of ${\cal A}^{\st}$ if $\pi$ has length $n$. 

When $\st$ is a descent statistic, the notion of $\st$-equivalence of permutations induces the notion of $\st$-equivalence of compositions as defined in the introduction, so we can think of the basis elements of ${\cal A}^{\st}$ as being $\st$-equivalence classes of compositions. From this perspective, it is evident from the product formula (\ref{e-Fprod}) that $\mathcal{A}^{\Des}$ is isomorphic to $\QSym$ with the basis of $\Des$-equivalence classes corresponding to the fundamental quasisymmetric functions.

The following provides a necessary and sufficient condition for a descent statistic to be shuffle-compatible.

\begin{thm}[{Gessel\textendash Zhuang \cite[Theorem 4.3]{Gessel2018}}] \label{t-gzmain}
A descent statistic $\st$ is shuffle-compatible if and only if there exists a $\mathbb{Q}$-algebra homomorphism $\phi_{\st}\colon\mathrm{QSym}\rightarrow A$, where $A$ is a $\mathbb{Q}$-algebra with basis $\{u_{\alpha}\}$ indexed by $\st$-equivalence classes $\alpha$ of compositions, such that $\phi_{\st}(F_{L})=u_{\alpha}$ whenever $L$ is in the $\st$-equivalence class $\alpha$. When this is the case, the map $u_{\alpha}\mapsto\alpha$ is a $\mathbb{Q}$-algebra isomorphism from $A$ to $\mathcal{A}^{\st}$.
\end{thm}

It follows from Theorem \ref{t-gzmain} that, whenever $\st$ is a shuffle-compatible descent statistic, the linear map $p_{\st}\colon\QSym\rightarrow\mathcal{A}^{\st}$ sending each $F_{L}$ to the $\st$-equivalence class of $L$ is a $\mathbb{Q}$-algebra homomorphism with the same kernel as $\phi_{\st}$; this implies that $\QSym/\ker p_{\st}$ is isomorphic to $\mathcal{A}^{\st}$ as algebras. Note that when $\st$ is not shuffle-compatible, it still holds that $\QSym/\ker p_{\st}$ and $\mathcal{A}^{\st}$ are isomorphic as vector spaces.

Gessel and Zhuang \cite{Gessel2018} used Theorem \ref{t-gzmain} to give explicit descriptions of the shuffle algebras of a number of descent statistics\textemdash including $\des$, $\maj$, $\Pk$, $\pk$, and $(\pk,\des)$\textemdash which yield algebraic proofs for their shuffle-compatibility. In \cite{Grinberg2018}, Grinberg proved Gessel and Zhuang's conjecture that the exterior peak set $\Epk$ is shuffle-compatible and gave a characterization of its shuffle algebra, introduced a strengthening of shuffle-compatibility called ``LR-shuffle-compatibility'' which is closely related to the dendriform algebra structure of $\QSym$, and initiated the study of the kernels $\mathcal{K}^{\st}$.

Recall that the \textit{kernel} $\mathcal{K}^{\st}$ of a descent statistic $\st$ is the subspace
\[
\mathcal{K}^{\st}=\Span\{\,F_{L}-F_{K}:L\sim_{\st}K\,\}
\]
of $\QSym$. Later on, we shall use the notation $\mathcal{K}_{n}^{\st}$ for the $n$th homogeneous component of $\mathcal{K}^{\st}$, so that
\[
\mathcal{K}_{n}^{\st}=\Span\{\,F_{L}-F_{K}:L\sim_{\st}K\text{ and }L,K\vDash n\,\}
\]
and
\[
\mathcal{K}^{\st}=\bigoplus_{n=0}^{\infty}\mathcal{K}_{n}^{\st}.
\]
It is easy to see that $\mathcal{K}^{\st}$ is precisely the kernel of the linear map $p_{\st}$ defined above, hence the name ``kernel''. The following is then a consequence of Theorem \ref{t-gzmain}.

\begin{thm}[{Grinberg \cite[Proposition 101]{Grinberg2018}}] \label{t-ideal}
A descent statistic $\st$ is shuffle-compatible if and only if $\mathcal{K}^{\st}$ is an ideal of $\QSym$. When this is the case, $\mathcal{A}^{\st}$ is isomorphic to $\QSym/\mathcal{K}^{\st}$ as $\mathbb{Q}$-algebras.
\end{thm}

\subsection{\label{ss-tri}Linear expansions and triangularity}

Lastly, we state a lemma concerning ``invertibly triangular expansions'' which will be used in our proof of Theorem \ref{t-M}. This lemma appears (in a slightly different yet equivalent form) in the Appendix to \cite{Grinberg2020}, which gives a treatment of some fundamental results from linear algebra for matrices whose rows and columns are indexed by arbitrary objects (rather than numbers). In particular, we are interested in the case when both the rows and columns are indexed by elements of a finite poset $S$, such that the matrix is ``invertibly triangular''\textemdash i.e., all the entries $a_{s,s}$ on the ``diagonal'' are invertible and $a_{s,t}=0$ whenever we do not have $t\leq s$. (When working over a field, as we do, the condition that $a_{s,s}$ is invertible is equivalent to $a_{s,s}\neq0$.) Note that this reduces to the typical notion of an invertible lower-triangular $n\times n$ matrix upon taking $S=[n]$.

A \textit{family} $(f_{s})_{s\in S}$ indexed by a set $S$ refers to an assignment of an object $f_{s}$ to each $s\in S$. The objects in a family need not be distinct\textemdash i.e., we may have $f_{s}=f_{t}$ for $s\neq t$. Roughly speaking, a family $(e_{s})_{s\in S}$ can be expanded invertibly triangularly in another family $(f_{s})_{s\in S}$ if we can write the $e_{s}$ as linear combinations of the $f_{s}$ such that the coefficients of these linear combinations form an invertibly triangular matrix in the sense described above. However, rather than giving the formal definition of an invertibly triangular expansion in terms of these generalized matrices, it is easier for our purposes to give the following equivalent definition. 

Given a $\mathbb{Q}$-vector space $V$, a finite poset $S$, and two families $(e_{s})_{s\in S}$ and $(f_{s})_{s\in S}$ of elements of $V$, we say that $(e_{s})_{s\in S}$ \textit{expands invertibly triangularly} in $(f_{s})_{s\in S}$ if, for each $s\in S$, we can write $e_{s}$ as
\begin{equation}
e_{s}=c_{s}f_{s}+\sum_{t>s}c_{t}f_{t} \label{e-invtri}
\end{equation}
for some $c_{t}\in\mathbb{Q}$ with $c_{s}\neq0$. Importantly, two families have the same span if one can be expanded invertibly triangularly in another.

\begin{lem}[{Grinberg\textendash Reiner \cite[Corollary 11.1.19 (b)]{Grinberg2020}}]\label{l-invtri}
Let $V$ be a $\mathbb{Q}$-vector space, $S$ a finite poset, and $(e_{s})_{s\in S}$ and $(f_{s})_{s\in S}$ two families of elements of $V$. If $(e_{s})_{s\in S}$ expands invertibly triangularly in $(f_{s})_{s\in S}$, then $\Span(e_{s})_{s\in S}=\Span(f_{s})_{s\in S}$.
\end{lem}

We note that our definition of ``expands invertibly triangularly'' is the opposite of that in \cite{Grinberg2020}---i.e., \cite{Grinberg2020} has $t<s$ in place of $t>s$ in \eqref{e-invtri}---but it is clear that Lemma \ref{l-invtri} still holds as we can simply reverse the order of the poset $S$.

\section{\label{s-fun}Characterizations in terms of the fundamental basis}

The purpose of this section is to provide proofs for Theorems \ref{t-Fsst} and \ref{t-F}. In proving Theorem~\ref{t-Fsst}, we will first establish a couple general results for arbitrary vector spaces.

\subsection{\label{ss-Fsst}Proof of Theorem \ref{t-Fsst}}

Throughout this section, fix a field $\Bbbk$ of characteristic $\neq 2$ and let $V$ be a vector space over $\Bbbk$ with basis $\{u_{s}\}_{s\in I}$ (where $I$ is an index set for the basis). Let $\sim$ be an equivalence relation on $I$, and let 
\[
R\coloneqq\{\,(s,t)\in I^{2}:s\sim t\,\}\quad\text{and}\quad\mathcal{U}\coloneqq\{\,u_{s}-u_{t}:s\sim t\,\}.
\]
Consider the subspace $W$ of $V$ spanned by the set $\mathcal{U}$. Note that every subset of $\mathcal{U}$ can be written as 
\[
\mathcal{U}_{S}\coloneqq\{\,u_{s}-u_{t}:(s,t)\in S\,\}
\]
for some $S\subseteq R$. We associate to $S$ a directed graph $G_{S}$ with vertex set $I$ and edge set $S$\textemdash i.e., there is an edge from $s$ to $t$ if and only if $(s,t)\in S$. As a special case, let $G\coloneqq G_{R}$, so the connected components of $G$ are simply the equivalence classes of $I$ under $\sim$. More generally, the connected components of $G_{S}$ refine the equivalence classes of $I$.

\begin{lem} \label{l-ccequiv}
Let $S\subseteq R$. If $s,t\in I$ are in the same connected component of $G_{S}$, then $s\sim t$.
\end{lem}

\begin{proof}
Let $s,t\in I$ be in the same connected component of $G_{S}$. Then there is a sequence 
\[
s=s_{0}\leftrightarrow s_{1}\leftrightarrow\cdots\leftrightarrow s_{k}=t
\]
where $s_{i-1}\leftrightarrow s_{i}$ means that there is an edge in either direction (not necessarily both) between $s_{i-1}$ and $s_{i}$. As a result, we have either $(s_{i-1},s_{i})\in S$ or $(s_{i},s_{i-1})\in S$ for each $i\in[k]$, which means that $s_{i-1}\sim s_{i}$ for each $i\in[k]$ and therefore $s\sim t$ by transitivity.
\end{proof}

\begin{lem}
\label{l-dg}Let $s,t\in I$ and $S\subseteq R$. Then $u_{s}-u_{t}\in\Span\mathcal{U}_{S}$ if and only if $s$ and $t$ are in the same connected component of $G_{S}$. 
\end{lem}

\begin{proof}
Suppose that $s$ and $t$ are in the same connected component of $G_{S}$. Then, as in the proof of Lemma \ref{l-ccequiv}, there is a sequence
\[
s=s_{0}\leftrightarrow s_{1}\leftrightarrow\cdots\leftrightarrow s_{k}=t
\]
where $s_{i-1}\leftrightarrow s_{i}$ means that there is an edge in either direction (not necessarily both) between $s_{i-1}$ and $s_{i}$. For each $i\in[k]$, we have $s_{i-1}\leftrightarrow s_{i}$ which means that $u_{s_{i-1}}-u_{s_{i}}\in\mathcal{U}_{S}$ or $u_{s_{i}}-u_{s_{i-1}}\in\mathcal{U}_{S}$. Thus $u_{s_{i-1}}-u_{s_{i}}\in\Span\mathcal{U}_{S}$ for each $i\in[k]$, which implies 
\[
u_{s}-u_{t}=u_{s_{0}}-u_{s_{k}}=\sum_{i=1}^{k}(u_{s_{i-1}}-u_{s_{i}})\in\Span\mathcal{U}_{S}.
\]

Conversely, suppose that $u_{s}-u_{t}\in\Span\mathcal{U}_{S}$. Let
$f_{s}\colon V\rightarrow\Bbbk$ be the linear map defined by 
\[
f_{s}(u_{r})=\begin{cases}
1, & \text{if }s\text{ and }r\text{ are in the same connected component of }G_{S},\\
0, & \text{otherwise}.
\end{cases}
\]
Recall that if $u_{p}-u_{q}\in\mathcal{U}_{S}$, then $p$ and $q$ are joined by an edge, which means $f_{s}(u_{p})=f_{s}(u_{q})$; after all, $s$ is either in the same connected component of $G_{S}$ as $p$ and $q$ or it is not. Thus, all of the spanning vectors $u_{p}-u_{q}$ of $\mathcal{U}_{S}$ belong to the kernel of $f_{s}$, so $\Span\mathcal{U}_{S} \subseteq \ker f_{s}$. Since $u_{s}-u_{t}\in\Span\mathcal{U}_{S}$, it follows that $u_{s}-u_{t}\in\ker f_{s}$. Hence, $f_{s}(u_{t})=f_{s}(u_{s})=1$ which concludes the proof.
\end{proof}

\begin{thm} \label{t-span}
Let $S\subseteq R$. Then $W$ is spanned by $\,\mathcal{U}_{S}$ if and only if the connected components of $G_{S}$ are precisely the connected components of $G$ \textup{(}as sets of vertices, not as subgraphs\textup{)}.
\end{thm}

\begin{proof}
Suppose that the connected components of $G_{S}$ are those of $G$. Fix $s,t\in I$ satisfying $s\sim t$. Then $s$ and $t$ belong to the same connected component of $G$ and thus the same connected component of $G_{S}$, so $u_{s}-u_{t}\in\Span\mathcal{U}_{S}$ by Lemma \ref{l-dg}. Since $u_{s}-u_{t}$ was arbitrarily taken from $\mathcal{U}$, which spans $W$, it follows by linearity that $W$ is spanned by $\mathcal{U}_{S}$.

Conversely, suppose that $W$ is spanned by $\mathcal{U}_{S}$. To show that $G_{S}$ and $G$ have the same connected components, it suffices to show that for any $s,t\in I$, if $s$ and $t$ are in the same connected component of $G$ then they are in the same connected component of $G_{S}$. As such, fix $s,t\in I$ which belong to the same connected component of $G$. Then $s\sim t$, and so $u_{s}-u_{t}\in\mathcal{U}\subseteq W=\Span\mathcal{U}_{S}$. By Lemma \ref{l-dg}, $s$ and $t$ are in the same connected component of $G_{S}$, and we are done.
\end{proof}

\begin{thm} \label{t-li}
Let $S\subseteq R$. Then $\mathcal{U}_{S}$ is linearly independent if and only if $G_{S}$ is a forest.
\end{thm}

\begin{proof}
Suppose that $G_{S}$ is a forest, and assume toward contradiction that $\mathcal{U}_{S}$ is linearly dependent. Then there exists $(s,t)\in S$ for which $u_{s}-u_{t}$ is a linear combination of other elements of $\mathcal{U}_{S}$. Let $S^{\prime}\coloneqq S\backslash\{(s,t)\}$, so that $u_{s}-u_{t}\in\Span\mathcal{U}_{S^{\prime}}$. Applying Lemma \ref{l-dg} to $S^{\prime}$, it follows that there is a path from $s$ to $t$ in the underlying undirected graph of $G_{S^{\prime}}$. Now, recall that there is an edge between $s$ and $t$ in $G_{S}$ because $(s,t)\in S$; combining this edge with the path from $s$ to $t$ in $G_{S^{\prime}}$ yields a cycle in $G_{S}$, which contradicts $G_{S}$ being a forest. Therefore, $\mathcal{U}_{S}$ is linearly independent.

Conversely, suppose that $G_{S}$ is not a forest. Then the undirected graph $G_{S}$ has a cycle $(s_{1},s_{2},\dots,s_{k})$, so each of the vectors 
\[
u_{s_{1}}-u_{s_{2}},u_{s_{2}}-u_{s_{3}},\dots,u_{s_{k-1}}-u_{s_{k}},u_{s_{k}}-u_{s_{1}}
\]
belongs to $\mathcal{U_{S}}$ up to sign (either it belongs to $\mathcal{U_{S}}$ or its negative does). However, the sum of these vectors is equal to 0, so $\mathcal{U}_{S}$ is linearly dependent.
\end{proof}

We can now recover Theorem \ref{t-Fsst} from Theorems \ref{t-span} and \ref{t-li}.

\begin{proof}[Proof of Theorem \ref{t-Fsst}]
Take $\Bbbk=\mathbb{Q}$, $V=\QSym$, $I=\mathcal{C}$, the basis $\{u_{s}\}$ to be the fundamental basis of $\QSym$, and $\sim$ to be $\st$-equivalence of compositions. Part (a) then follows from Theorem \ref{t-span}, and part (b) from Theorem \ref{t-li}.
\end{proof}

\subsection{\label{ss-FPk}Application to the peak set\textemdash proof of Theorem \ref{t-F} (a)}

We shall now apply Theorem \ref{t-Fsst} to obtain spanning sets for the kernels $\mathcal{K}^{\Pk}$ and $\mathcal{K}^{\pk}$ as given in Theorem \ref{t-F}. Let us begin with the peak set statistic. Recall the notations $\rightarrow_{1}$ and $\rightarrow_{2}$ defined in the introduction; for convenience, let us write $J\rightarrow_{\Pk}K$ if $J\rightarrow_{1}K$ or $J\rightarrow_{2}K$.

\begin{lem} \label{l-Pkequiv}
If $J\rightarrow_{\Pk}K$, then $J$ and $K$ are $\Pk$-equivalent.
\end{lem}

\begin{proof}
It is immediate from the definitions of $\rightarrow_{1}$ and $\rightarrow_{2}$ that $J\rightarrow_{\Pk}K$ implies $\left|J\right|=\left|K\right|$. So, it remains to show that $J\rightarrow_{\Pk}K$ implies $\Pk J=\Pk K$. Suppose that $J\rightarrow_{1}K$, so that there exists $l\in[m]$ for which $j_{l}>2$ and 
\[
K=(j_{1},\dots,j_{l-1},1,j_{l}-1,j_{l+1},\dots,j_{m}).
\]
Since $j_{l}>2$, we have that $j_{l}-1\geq2$ and thus $\Pk J=\Pk K$ by Lemma \ref{l-PkpkL} (a). The proof for the case $J\rightarrow_{2}K$ is similar.
\end{proof}

We are now ready to prove Theorem \ref{t-F} (a).

\begin{proof}[Proof of Theorem \ref{t-F} \textup{(}a\textup{)}]

We shall apply Theorem \ref{t-Fsst} (a) to $\st=\Pk$. By Lemma \ref{l-Pkequiv}, we may take $S=\{\,(J,K):J\rightarrow_{\Pk}K\,\}$. We seek to show that $J\sim_{\Pk}K$ if and only if $J$ and $K$ are in the same connected component of $G_{S}$, and in light of Lemma \ref{l-ccequiv}, it remains to prove the forward direction.

Suppose that $J\sim_{\Pk}K$. Let $J_{\alpha}$ be the composition obtained from $J$ by replacing each part $j_{l}>2$ with the parts of the composition $(1^{j_{l}-2},2)\vDash j_{l}$. For example, if $J=(3,2,4,1)$ then $J_{\alpha}=(1,2,2,1,1,2,1)$. Then there is a sequence of compositions satisfying
\[
J\rightarrow_{1}J_{1}\rightarrow_{1}J_{2}\rightarrow_{1}\cdots\rightarrow_{1}J_{\alpha}.
\]
Continuing the example above, we have 
\[
J=(3,2,4,1)\rightarrow_{1}(1,2,2,4,1)\rightarrow_{1}(1,2,2,1,3,1)\rightarrow_{1}(1,2,2,1,1,2,1)=J_{\alpha}.
\]
It follows that $J$ and $J_{\alpha}$ are in the same connected component of $G_{S}$, and similarly with $K$ and $K_{\alpha}$\@.

Next, if 2 is the final part of $J_{\alpha}$, then let $J_{\beta}$ be the composition obtained from $J_{\alpha}$ by replacing the final part 2 with two 1s. For instance, if $J_{\alpha}=(1,2,1,2,2)$ then $J_{\beta}=(1,2,1,2,1,1)$. Otherwise, if $J_{\alpha}$ ends with a 1, then set $J_{\beta}\coloneqq J_{\alpha}$. Note that either $J_{\alpha}=J_{\beta}$ or $J_{\alpha}\rightarrow_{2}J_{\beta}$; either way, $J_{\alpha}$ is in the same connected component as $J_{\beta}$, and thus $J$ is as well. The same holds for $K$ and $K_{\beta}$. We then have $J_{\beta}\sim_{\Pk}J\sim_{\Pk}K\sim_{\Pk}K_{\beta}$ by Lemma \ref{l-ccequiv}, so $\Pk J_{\beta}=\Pk K_{\beta}$ and $\left|J_{\beta}\right|=\left|K_{\beta}\right|$.

We claim that $J_{\beta}=K_{\beta}$; this will imply that $J$ and $K$ are in the same connected component. Assume by contradiction that $J_{\beta}\neq K_{\beta}$. Let $l$ be the position of the first part where $J_{\beta}$ and $K_{\beta}$ differ. Since $J_{\beta}$ and $K_{\beta}$ have all parts 1 or 2, we may assume without loss of generality that the $l$th part of $J_{\beta}$ is a 1 and that the $l$th part of $K_{\beta}$ is a 2. Note that the $l$th part of $J_{\beta}$ or $K_{\beta}$ cannot be its final part; after all, $K_{\beta}$ cannot end with a 2 by construction, and if the $l$th part of $J_{\beta}$ is its final part then we would have $\left|J_{\beta}\right|<\left|K_{\beta}\right|$. Thus, Lemma \ref{l-PkpkL} (a) implies that the sum of the first $l$ parts of $K_{\beta}$ is an element of $\Pk K_{\beta}$ but not of $\Pk J_{\beta}$, which contradicts $\Pk J_{\beta}=\Pk K_{\beta}$. Therefore, $J_{\beta}=K_{\beta}$.

We have shown that the connected components of $G_{S}$ are precisely the $\Pk$-equivalence classes of compositions; hence the result follows from Theorem \ref{t-Fsst} (a).
\end{proof}

The spanning set for $\mathcal{K}^{\Pk}$ provided by Theorem \ref{t-F} (a) is not a basis for $\mathcal{K}^{\Pk}$, as the corresponding directed graph is not a forest. For example, upon revisiting Figure \ref{f-arrows}, we see that there is a cycle formed by the relations
\[
(3,2)\rightarrow_{1}(1,2,2)\rightarrow_{2}(1,2,1,1) \,\,_{1}\!\!\leftarrow(3,1,1) \,\,_{2} \!\!\leftarrow(3,2).
\]
One way that we can obtain a basis from Theorem \ref{t-F} (a) is to specify that we must ``split'' the first entry that we are able to. Then, for example, we no longer have $(3,2)\rightarrow_{2}(3,1,1)$ because we are forced to split the first part, leading to $(1,2,2)$. More formally, we have the following: Given compositions $J=(j_{1},j_{2},\dots,j_{m})$ and $K$, let us write $J\rightarrowtriangle_{\Pk}K$ if 
\[
l=\min\{\,i:j_{i}>2,\text{ or }i=m\text{ and }j_{m}=2\,\}
\]
exists and 
\[
K=(j_{1},\dots,j_{l-1},1,j_{l}-1,j_{l+1},\dots,j_{m}).
\]

\begin{thm} \label{t-Pkbasis} 
The set $\{\,F_{J}-F_{K}:J\rightarrowtriangle_{\Pk}K\,\}$ is a linear basis of the ideal $\mathcal{K}^{\Pk}$.
\end{thm}

\begin{proof}
Let $S=\{\,(J,K):J\rightarrowtriangle_{\Pk}K\,\}$, so that we wish to show that $\mathcal{F}_{S}^{\Pk}$ is a basis of $\mathcal{K}^{\Pk}$. We shall first argue that $G_{S}$ is a forest. First, observe that $G_{S}$ has no directed cycles; after all, $J\rightarrowtriangle_{\Pk}K$ implies that $K$ has more parts than $J$. Therefore, $G_{S}$ can only fail to be a forest if it has an undirected cycle which is not a directed cycle, and it is not hard to see that this is only possible if two edges of $G_{S}$ have the same tail. (Given an undirected cycle, any attempt to orient the edges of this cycle will either result in a directed cycle or a vertex with outdegree 2.) By definition of $\rightarrowtriangle_{\Pk}$, for every composition $J$ there is at most one $K$ for which $J\rightarrowtriangle_{\Pk}K$, which means that no vertex of $G_{S}$ is the tail of more than one edge. Therefore, $G_{S}$ is a forest which implies that $\mathcal{F}_{S}^{\Pk}$ is linearly independent by Theorem \ref{t-Fsst} (b).

The proof of Theorem \ref{t-F} (a) can be easily adapted to show that $\mathcal{F}_{S}^{\Pk}$ is a spanning set for $\mathcal{K}^{\Pk}$ as well. Indeed, we already have $J_{\alpha}\rightarrowtriangle_{\Pk}J_{\beta}$ and $K_{\alpha}\rightarrowtriangle_{\Pk}K_{\beta}$, and when forming the sequences $J\rightarrow_{1}J_{1}\rightarrow_{1}J_{2}\rightarrow_{1}\cdots\rightarrow_{1}J_{\alpha}$ and $K\rightarrow_{1}K_{1}\rightarrow_{1}K_{2}\rightarrow_{1}\cdots\rightarrow_{1}K_{\alpha}$, we can require that the parts be split from left to right so that
\[
J\rightarrowtriangle_{\Pk}J_{1}\rightarrowtriangle_{\Pk}J_{2}\rightarrowtriangle_{\Pk}\cdots\rightarrowtriangle_{\Pk}J_{\alpha}\quad\text{and}\quad K\rightarrowtriangle_{\Pk}K_{1}\rightarrowtriangle_{\Pk}K_{2}\rightarrowtriangle_{\Pk}\cdots\rightarrowtriangle_{\Pk}K_{\alpha}.
\]
The rest of the proof proceeds in the same way.
\end{proof}

\subsection{\label{ss-Fpk}Application to the peak number\textemdash proof of Theorem \ref{t-F} (b)}

Similar to the notation $\rightarrow_{\Pk}$ defined in Section \ref{ss-FPk}, let us write $J\rightarrow_{\pk}K$ if $J\rightarrow_{1}K$, $J\rightarrow_{2}K$, or $J\rightarrow_{3}K$.

\begin{lem} \label{l-pkequiv}
If $J\rightarrow_{\pk}K$, then $J$ and $K$ are $\pk$-equivalent.
\end{lem}

\begin{proof}
If $J\rightarrow_{1}K$ or $J\rightarrow_{2}K$, then it follows from Lemma \ref{l-Pkequiv} that $J$ and $K$ are $\Pk$-equivalent and thus $\pk$-equivalent. If $J\rightarrow_{3}K$, then $K$ has the same parts as $J$ but listed in a different order, which implies $\pk J=\pk K$ by Lemma \ref{l-PkpkL} (b) as well as $\left|J\right|=\left|K\right|$.
\end{proof}

We now proceed to the proof of Theorem \ref{t-F} (b).

\begin{proof}[Proof of Theorem \ref{t-F} \textup{(}b\textup{)}]
We follow the same approach taken in the proof of part (a). Let $S=\{\,(J,K):J\rightarrow_{\pk}K\,\}$, which is a subset of $\{\,(J,K):J\sim_{\pk}K\,\}$ by Lemma \ref{l-pkequiv}. From Lemma \ref{l-ccequiv}, we know that if $J$ and $K$ are in the same connected component of $G_{S}$ then $J\sim_{\pk}K$, so it remains to prove the converse.

Suppose that $J\sim_{\pk}K$. Define $J_{\beta}$ and $K_{\beta}$ in the same way as in the proof of part (a); as before, $J$ and $J_{\beta}$ are in the same connected component and the same is true for $K$ and $K_{\beta}$. Recall that $J_{\beta}$ and $K_{\beta}$ have all parts 1 and 2, and end with a 1. Let $J_{\gamma}$ be the composition $(2^{b},1^{a})$ where $a>0$ is the number of 1s in $J_{\beta}$ and $b$ is the number of 2s in $J_{\beta}$. Then there is a sequence of compositions satisfying
\[
J_{\beta}\rightarrow_{3}J_{1}\rightarrow_{3}J_{2}\rightarrow_{3}\cdots\rightarrow_{3}J_{\gamma}.
\]
For example, if $J_{\beta}=(1,2,2,1,1)$, then 
\[
J_{\beta}=(1,2,2,1,1)\rightarrow_{3}(2,1,2,1,1)\rightarrow_{3}(2,2,1,1,1)=J_{\gamma}.
\]
Note that $J_{\gamma}$ is in the same connected component of $G_{S}$ as $J_{\beta}$ and thus $J$, and similarly with $K_{\gamma}$ and $K$. Therefore, $J\sim_{\pk}J_{\gamma}$ and $K\sim_{\pk}K_{\gamma}$. Along with $J\sim_{\pk}K$, these $\pk$-equivalences imply that $J_{\gamma}\sim_{\pk}K_{\gamma}$, so $\pk J_{\gamma}=\pk K_{\gamma}$ and $\left|J_{\gamma}\right|=\left|K_{\gamma}\right|$. From Lemma \ref{l-PkpkL} (b), we know that the number of peaks is equal to the number of non-final parts of size at least 2, and since both $J_{\gamma}$ and $K_{\gamma}$ consist of a sequence of 2s followed by a sequence of 1s, we conclude that $J_{\gamma}=K_{\gamma}$. Hence, $J$ and $K$ are in the same connected component of $G_{S}$.

Since the connected components of $G_{S}$ are precisely the $\pk$-equivalence classes of compositions, applying Theorem \ref{t-Fsst} (a) yields the desired result.
\end{proof}

Like with $\mathcal{K}^{\Pk}$, we shall trim the spanning set for $\mathcal{K}^{\pk}$ provided by Theorem \ref{t-F} (b) to yield a basis for this ideal. Given compositions $J=(j_{1},j_{2},\dots,j_{m})$ and $K$, we write $J\rightarrowtriangle_{\pk}K$ if either $J\rightarrowtriangle_{\Pk}K$\textemdash or, in the case that $l$ (as in the definition of $\rightarrowtriangle_{\Pk}$) does not exist, if
\[
k=\min\{\,i:j_{i}=1\text{ and }j_{i+1}=2\,\}
\]
exists and 
\[
K=(j_{1},\dots,j_{k-1},j_{k+1},j_{k},j_{k+2},\dots,j_{m}).
\]
Informally speaking, here we are also requiring the ``swapping'' ($\rightarrow_{3}$) to occur from left to right and only after all of the ``splitting'' ($\rightarrow_{1}$ and $\rightarrow_{2}$) has taken place. 

Given a composition $J=(j_1, j_2, \dots, j_m)$, let us call the pair $(j_k,j_l)$ an \textit{inversion} of the composition $J$ if $1 \leq k<l \leq m$ and $j_k > j_l$.

\begin{thm} \label{t-pkbasis}
The set $\{\,F_{J}-F_{K}:J\rightarrowtriangle_{\pk}K\,\}$ is a linear basis of the ideal $\mathcal{K}^{\pk}$.
\end{thm}

\begin{proof}

Let $S=\{\,(J,K):J\rightarrowtriangle_{\pk}K\,\}$; we wish to show that $\mathcal{F}_{S}^{\pk}$ is a basis of $\mathcal{K}^{\pk}$. By the same reasoning as in the proof of Theorem \ref{t-Pkbasis}, $G_{S}$ cannot have a directed cycle containing any edges of the form $J \rightarrowtriangle_{\Pk} K$. On the other hand, we also cannot have a directed cycle only containing edges of the form $J\rightarrow_{3}K$, as when $J\rightarrow_{3}K$, the number of inversions of $K$ is greater than that of $J$. Therefore, $G_{S}$ does not have any directed cycles. The same reasoning from the proof of Theorem \ref{t-Pkbasis} shows that $G_{S}$ has no undirected cycles either. Thus, $G_{S}$ is a forest, and it follows from Theorem \ref{t-Fsst} (b) that $\mathcal{F}_{S}^{\pk}$ is linearly independent.

The proof that $\mathcal{F}_{S}^{\pk}$ spans $\mathcal{K}^{\pk}$ is similar to that of the analogous result for $\Pk$ (again, see the proof of Theorem \ref{t-Pkbasis}); we omit the details.
\end{proof}

\section{\label{s-mon}Characterizations in terms of the monomial basis}

\subsection{\label{ss-MPk}The peak set\textemdash proof of Theorem \ref{t-M}
(a)}

We now turn our attention to characterizing the ideals $\mathcal{K}^{\Pk}$ and $\mathcal{K}^{\pk}$ in terms of the monomial basis, beginning with $\mathcal{K}^{\Pk}$. Our proofs here are based on Grinberg's proof for the analogous result on $\mathcal{K}^{\Epk}$. Let us first introduce some notation which will simplify the presentation of our proofs.

Recall that $\mathcal{K}_{n}^{\st}$ denotes the $n$th homogeneous component of $\mathcal{K}^{\st}$. Define 
\begin{align*}
\mathcal{F}_{n}^{\Pk} & \coloneqq\{\,F_{J}-F_{K}:J\rightarrow_{\Pk}K\text{ and }J,K\vDash n\,\}\quad\text{and}\\
\mathcal{M}_{n}^{\Pk} & \coloneqq\{M_{J}+M_{K}:J\triangleright_{1}K\text{ or }J\triangleright_{2}K,\text{ and }J,K\vDash n\}\cup\{M_{(1^{n-2},2)}\}.
\end{align*}
We know from Theorem \ref{t-F} (a) that $\mathcal{K}_{n}^{\Pk}=\Span\mathcal{F}_{n}^{\Pk}$ for all $n$, and our present goal is to show that $\mathcal{K}_{n}^{\Pk}=\Span\mathcal{M}_{n}^{\Pk}$ for all $n$.

Let
\begin{align*}
\Omega_{n,1} & \coloneqq\{\,(C,k):C\subsetneq[n-1],\,k\notin C,\,k-1\notin C,\,k-2\in C\cup\{0\}\,\},\\
\Omega_{n,2} & \coloneqq\{\,(C,k):C\subsetneq[n-2],\,n-2\in C,\,k\notin C,\,k+1\in C,\,k-1\in C\cup\{0\}\,\},\text{ and}\\
\Omega_{n,3} & \coloneqq\{\,([n-2],n-1)\,\}
\end{align*}
be subsets of $2^{[n-1]}\times[n-1]$, where $2^{[n-1]}$ is the power set of $[n-1]$. Here we let $[0]$ be the empty set, so that $\Omega_{2,3}=\{(\emptyset,1)\}$. Note that $\Omega_{n,1}$ is empty for $n\leq2$, $\Omega_{n,2}$ is empty for $n\leq3$, and $\Omega_{n,3}$ is empty for $n\leq1$. It is easy to check that the sets $\Omega_{n,1}$, $\Omega_{n,2}$, and $\Omega_{n,3}$ are disjoint for any fixed $n$; let
\[
\Omega_{n}\coloneqq\Omega_{n,1}\sqcup\Omega_{n,2}\sqcup\Omega_{n,3}.
\]

Given $(C,k)\in\Omega_{n}$, define the quasisymmetric functions $\mathbf{f}_{C,k}$ and $\mathbf{m}_{C,k}$\footnote{The $\mathbf{f}_{C,k}$ and $\mathbf{m}_{C,k}$ depend on $n$, but we suppress $n$ from their notation for the sake of simplicity.} by
\[
\mathbf{f}_{C,k}=\begin{cases}
F_{n,C}-F_{n,C\cup\{k-1\}}, & \text{if }(C,k)\in\Omega_{n,1},\\
F_{n,C}-F_{n,C\cup\{n-1\}}, & \text{if }(C,k)\in\Omega_{n,2}\text{ or }(C,k)\in\Omega_{n,3},
\end{cases}
\]
and 
\[
\mathbf{m}_{C,k}=\begin{cases}
M_{n,C}+M_{n,C\cup\{k\}}, & \text{if }(C,k)\in\Omega_{n,1}\text{ or }(C,k)\in\Omega_{n,2},\\
M_{n,C}, & \text{if }(C,k)\in\Omega_{n,3}.
\end{cases}
\]
In what follows, we will consider the families $(\mathbf{f}_{C,k})_{(C,k)\in\Omega_{n}}$ and $(\mathbf{m}_{C,k})_{(C,k)\in\Omega_{n}}$. As noted in Section \ref{ss-tri}, a family can admit repeated members. For instance, given $(C,k_{1}),(C,k_{2})\in\Omega_{n,2}$ where $k_{1}\neq k_{2}$, we have $\mathbf{f}_{C,k_{1}}=\mathbf{f}_{C,k_{2}}$. This is in contrast to $\mathcal{F}_{n}^{\Pk}$ and $\mathcal{M}_{n}^{\Pk}$, which have no repeated elements by virtue of being sets.

In order to prove Theorem \ref{t-M} (a), we first prove that $\Span(\mathbf{f}_{C,k})_{(C,k)\in\Omega_{n}}=\Span\mathcal{F}_{n}^{\Pk}$ and $\Span(\mathbf{m}_{C,k})_{(C,k)\in\Omega_{n}}=\Span\mathcal{M}_{n}^{\Pk}$, and then show that $\Span(\mathbf{f}_{C,k})_{(C,k)\in\Omega_{n}}=\Span(\mathbf{m}_{C,k})_{(C,k)\in\Omega_{n}}$. In doing so, we will make repeated use of the correspondence between compositions of $n$ and subsets of $[n-1]$, but many of the details will be omitted as they are straightforward to verify from the relevant definitions yet distract from the main essence of the proof.

\begin{prop} \label{p-PkftoF}
For all $n\geq0$, we have $\Span(\mathbf{f}_{C,k})_{(C,k)\in\Omega_{n}}=\Span\mathcal{F}_{n}^{\Pk}$.
\end{prop}

\begin{proof}
Fix $(C,k)\in\Omega_{n}$. The following are readily checked:
\begin{enumerate}
\item[(1)] If $(C,k)\in\Omega_{n,1}$, then $\Comp C\rightarrow_{1}\Comp(C\cup\{k-1\})$
and so taking $J=\Comp C$ and $K=\Comp(C\cup\{k-1\})$ yields
\[
\mathbf{f}_{C,k}=F_{n,C}-F_{n,C\cup\{k-1\}}=F_{J}-F_{K}\in\mathcal{F}_{n}^{\Pk}.
\]
\item[(2)] If $(C,k)\in\Omega_{n,2}$ or $(C,k)\in\Omega_{n,3}$, then $\Comp C\rightarrow_{2}\Comp(C\cup\{n-1\})$
and so taking $J=\Comp C$ and $K=\Comp(C\cup\{n-1\})$ implies
\[
\mathbf{f}_{C,k}=F_{n,C}-F_{n,C\cup\{n-1\}}=F_{J}-F_{K}\in\mathcal{F}_{n}^{\Pk}.
\]
\end{enumerate}
Thus $\mathbf{f}_{C,k}\in\mathcal{F}_{n}^{\Pk}$ for all $(C,k)\in\Omega_{n}$, which implies $\Span(\mathbf{f}_{C,k})_{(C,k)\in\Omega_{n}}\subseteq\Span\mathcal{F}_{n}^{\Pk}$ by linearity.

For the reverse inclusion, fix $J,K\vDash n$ satisfying $J\rightarrow_{\Pk}K$. As above, we consider separate cases:
\begin{enumerate}
\item[(1)] Suppose $J\rightarrow_{1}K$, so that we may write $J=(j_{1},j_{2},\dots,j_{m})$ and $K=(j_{1},\dots,j_{l-1},1,j_{l}-1,j_{l+1},\dots,j_{m})$ where $j_{l}>2$. Then $\Des K=(\Des J)\cup\{k-1\}$ where $k-1=j_{1}+\cdots+j_{l-1}+1$, and $(\Des J,k)\in\Omega_{n,1}$. Hence 
\[
F_{J}-F_{K}=F_{n,\Des J}-F_{n,(\Des J)\cup\{k-1\}}=\mathbf{f}_{\Des J,k}.
\]
\item[(2)] Suppose $J\rightarrow_{2}K$, and write $J=(j_{1},j_{2},\dots,j_{m-1},2)$
and $K=(j_{1},\dots,j_{m-1},1,1)$. Then $\Des K=(\Des J)\cup\{n-1\}$. We now consider three subcases:
\begin{enumerate}
\item Suppose that $j_{a}=2$ for some $1\leq a\leq m-1$, and let $k=j_{1}+\cdots+j_{a}-1$.
Then $(\Des J,k)\in\Omega_{n,2}$, so 
\[
F_{J}-F_{K}=F_{n,\Des J}-F_{n,(\Des J)\cup\{n-1\}}=\mathbf{f}_{\Des J,k}.
\]
\item Suppose that none of the parts $j_{1},\dots,j_{m-1}$ are equal to 2 but that $j_{a}>2$ for some $a\in[m-1]$. Let $J_{i}$ be the composition obtained from $J$ by replacing $j_{a}$ with the parts of the composition $(1^{i},j_{a}-i)$, and define $K_{i}$ in the same way. Then we have
\begin{alignat*}{1}
J & =J_{0}\rightarrow_{1}J_{1}\rightarrow_{1}J_{2}\rightarrow_{1}\cdots\rightarrow_{1}J_{j_{a}-2}\quad\text{and}\\
K & =K_{0}\rightarrow_{1}K_{1}\rightarrow_{1}K_{2}\rightarrow_{1}\cdots\rightarrow_{1}K_{j_{a}-2},
\end{alignat*}
and also $J_{j_{a}-2}\rightarrow_{2}K_{j_{a}-2}$. Furthermore, by the telescope principle, we have
\[
F_{J}-F_{K}=\sum_{i=1}^{j_{a}-2}(F_{J_{i-1}}-F_{J_{i}})+(F_{J_{j_{a}-2}}-F_{K_{j_{a}-2}})-\sum_{i=1}^{j_{a}-2}(F_{K_{i-1}}-F_{K_{i}}).
\]
From Case (1) above, we know that each of the terms $F_{J_{i-1}}-F_{J_{i}}$ and $F_{K_{i-1}}-F_{K_{i}}$ belongs to the family $(\mathbf{f}_{C,k})_{(C,K)\in\Omega_{n}}$, and from Case (2a), the same is true for $F_{J_{j_{a}-2}}-F_{K_{j_{a}-2}}$. Therefore, we have $F_{J}-F_{K}\in\Span(\mathbf{f}_{C,k})_{(C,K)\in\Omega_{n}}$.
\item Suppose that all of the parts $j_{1},\dots,j_{m-1}$ are equal to 1. Then $\Des J=[n-2]$ and $(\Des J,n-1)\in\Omega_{n,3}$, so 
\[
F_{J}-F_{K}=F_{n,[n-2]}-F_{n,[n-2]\cup\{n-1\}}=\mathbf{f}_{\Des J,n-1}.
\]
\end{enumerate}
\end{enumerate}
In all cases we have $F_{J}-F_{K}\in\Span(\mathbf{f}_{C,k})_{(C,K)\in\Omega_{n}}$, so $\Span\mathcal{F}_{n}^{\Pk}\subseteq\Span(\mathbf{f}_{C,k})_{(C,K)\in\Omega_{n}}$ by linearity. We conclude that the two spans are equal.
\end{proof}

\begin{prop} \label{p-PkmtoM}
For all $n\geq0$, we have $\Span(\mathbf{m}_{C,k})_{(C,k)\in\Omega_{n}}=\Span\mathcal{M}_{n}^{\Pk}$.
\end{prop}

\begin{proof}
Fix $(C,k)\in\Omega_{n}$. Then one can readily verify the following:
\begin{enumerate}
\item[(1)] If $(C,k)\in\Omega_{n,1}$, then $\Comp C\triangleright_{1}\Comp(C\cup\{k\})$ and so taking $J=\Comp C$ and $K=\Comp(C\cup\{k\})$ yields
\[
\mathbf{m}_{C,k}=M_{n,C}+M_{n,C\cup\{k\}}=M_{J}+M_{K}\in\mathcal{M}_{n}^{\Pk}.
\]
\item[(2)] If $(C,k)\in\Omega_{n,2}$, then $\Comp C\triangleright_{2}\Comp(C\cup\{k\})$ and so taking $J=\Comp C$ and $K=\Comp(C\cup\{k\})$ yields
\[
\mathbf{m}_{C,k}=M_{n,C}+M_{n,C\cup\{k\}}=M_{J}+M_{K}\in\mathcal{M}_{n}^{\Pk}.
\]
\item[(3)] If $(C,k)\in\Omega_{n,3}$, then $\Comp C=(1^{n-2},2)$ and so
\[
\mathbf{m}_{C,k}=M_{n,C}=M_{(1^{n-2},2)}\in\mathcal{M}_{n}^{\Pk}.
\]
\end{enumerate}
Since $\mathbf{m}_{C,k}\in\mathcal{M}_{n}^{\Pk}$ for all $(C,k)\in\Omega_{n}$, we have $\Span(\mathbf{m}_{C,k})_{(C,k)\in\Omega_{n}}\subseteq\Span\mathcal{M}_{n}^{\Pk}$ by linearity.

Conversely, fix $J,K\vDash n$ satisfying $J\triangleright_{1}K$ or $J\triangleright_{2}K$.
\begin{enumerate}
\item[(1)] Suppose $J\triangleright_{1}K$, so that we may write $J=(j_{1},j_{2},\dots,j_{m})$ and $K=(j_{1},\dots,j_{l-1},2,j_{l}-2,j_{l+1},\dots,j_{m})$ where $j_{l}>2$. Taking $k=j_{1}+\cdots+j_{l-1}+2$, we have $\Des K=(\Des J)\cup\{k\}$ and $(\Des J,k)\in\Omega_{n,1}$, so 
\[
M_{J}+M_{K}=M_{n,\Des J}+M_{n,(\Des J)\cup\{k\}}=\mathbf{m}_{\Des J,k}.
\]
\item[(2)] Suppose $J\triangleright_{2}K$, so that we may write $J=(j_{1},j_{2},\dots,j_{m-1},2)$ where $j_{l}=2$ for some $l\in[m-1]$, and $K=(j_{1},\dots,j_{l-1},1,1,j_{l+1},\dots,j_{m-1},2)$. Taking $k=j_{1}+\cdots+j_{l-1}+1$, we have $\Des K=(\Des J)\cup\{k\}$ and $(\Des J,k)\in\Omega_{n,2}$, so 
\[
M_{J}+M_{K}=M_{n,\Des J}+M_{n,(\Des J)\cup\{k\}}=\mathbf{m}_{\Des J,k}.
\]
\end{enumerate}
Moreover, if $J=(1^{n-2},2)$, then we have $\Des J=[n-2]$ and $(\Des J,n-1)\in\Omega_{n,3}$, which means $M_{J}=\mathbf{m}_{\Des J,n-1}$. We thus have $\Span\mathcal{M}_{n}^{\Pk}\subseteq\Span(\mathbf{m}_{C,k})_{(C,k)\in\Omega_{n}}$ by linearity, and the proof is complete.
\end{proof}

\begin{prop} \label{p-Pkftom}
For all $n\geq0$, we have $\Span(\mathbf{f}_{C,k})_{(C,k)\in\Omega_{n}}=\Span(\mathbf{m}_{C,k})_{(C,k)\in\Omega_{n}}$.
\end{prop}

\begin{proof}
Let us define a partial order on each $\Omega_{n,i}$ by setting $(B,k)\geq(C,l)$ if $k=l$ and $C\subseteq B$. Then, we endow $\Omega_{n}$ with the partial order obtained by taking the disjoint union of the posets $\Omega_{n,1}$, $\Omega_{n,2}$, and $\Omega_{n,3}$, so that $(B,k)$ and $(C,l)$ are incomparable if they are not in the same $\Omega_{n,i}$. By Lemma \ref{l-invtri}, it suffices to show that $(\mathbf{m}_{C,k})_{(C,k)\in\Omega_{n}}$ expands invertibly triangularly in $(\mathbf{f}_{C,k})_{(C,k)\in\Omega_{n}}$ with respect to this partial order.

First, fix $(C,k)\in\Omega_{n,1}$. It is straightforward to show that
\[
\{\,(B,k):C\subseteq B\subseteq[n-1],\,k-1\notin B,\,k\notin B\,\}=\{\,(B,l)\in\Omega_{n,1}:(B,l)\geq(C,k)\,\};
\]
by this set equality and Lemma \ref{l-MtoF} (c), we have
\begin{align*}
\mathbf{m}_{C,k} & =M_{n,C}+M_{n,C\cup\{k\}}\\
 & =\sum_{\substack{C\subseteq B\subseteq[n-1]\\
k-1,k\notin B
}
}(-1)^{\left|B\backslash C\right|}(F_{n,B}-F_{n,B\cup\{k-1\}})\\
 & =\sum_{\substack{(B,l)\in\Omega_{n,1}\\
(B,l)\geq(C,k)
}
}(-1)^{\left|B\backslash C\right|}\mathbf{f}_{B,l}\\
 & =\text{\ensuremath{\mathbf{f}_{C,k}}+}\sum_{\substack{(B,l)\in\Omega_{n}\\
(B,l)>(C,k)
}
}(-1)^{\left|B\backslash C\right|}\mathbf{f}_{B,l}.
\end{align*}

Second, fix $(C,k)\in\Omega_{n,2}$. By Lemma \ref{l-MtoF} (b), we have
\begin{align*}
\mathbf{m}_{C,k} & =M_{n,C}+M_{n,C\cup\{k\}}\\
 & =\sum_{\substack{C\subseteq B\subseteq[n-1]\\
k\notin B
}
}(-1)^{\left|B\backslash C\right|}F_{n,B}\\
 & =\sum_{\substack{C\subseteq B\subseteq[n-1]\\
k,n-1\notin B
}
}(-1)^{\left|B\backslash C\right|}F_{n,B}+\sum_{\substack{C\subseteq B\subseteq[n-1]\\
k\notin B,\,n-1\in B
}
}(-1)^{\left|B\backslash C\right|}F_{n,B}.
\end{align*}
Recall that that $(C,k)\in\Omega_{n,2}$ implies $k\neq n-1$. By toggling the inclusion of $n-1$ in $B$, we see that the sets $B$ satisfying $C\subseteq B\subseteq[n-1]$ and $k,n-1\notin B$ are in bijection with those satisfying $C\subseteq B\subseteq[n-1]$, $k\notin B$, and $n-1\in B$. Thus, we can write
\begin{align*}
\mathbf{m}_{C,k} & =\sum_{\substack{C\subseteq B\subseteq[n-1]\\
k,n-1\notin B
}
}(-1)^{\left|B\backslash C\right|}F_{n,B}+\sum_{\substack{C\subseteq B\subseteq[n-1]\\
k,n-1\notin B
}
}(-1)^{\left|(B\cup\{n-1\})\backslash C\right|}F_{n,B\cup\{n-1\}}\\
 & =\sum_{\substack{C\subseteq B\subseteq[n-1]\\
k,n-1\notin B
}
}(-1)^{\left|B\backslash C\right|}(F_{n,B}-F_{n,B\cup\{n-1\}}).
\end{align*}
A routine argument yields the set equality
\[
\{\,(B,k):C\subseteq B\subseteq[n-1],\,k\notin B,\,n-1\notin B\,\}=\{\,(B,l)\in\Omega_{n,2}:(B,l)\geq(C,k)\,\},
\]
whence
\[
\mathbf{m}_{C,k}=\sum_{\substack{(B,l)\in\Omega_{n,2}\\
(B,l)\geq(C,k)
}
}(-1)^{\left|B\backslash C\right|}\mathbf{f}_{B,l}=\mathbf{f}_{C,k}+\sum_{\substack{(B,l)\in\Omega_{n}\\
(B,l)>(C,k)
}
}(-1)^{\left|B\backslash C\right|}\mathbf{f}_{B,l}.
\]

Lastly, fix $(C,k)\in\Omega_{n,3}$, so that $C=[n-2]$ and $k=n-1$. Then from Lemma \ref{l-MtoF} (a), we obtain
\begin{align*}
\mathbf{m}_{C,k} & =M_{n,[n-2]}=F_{n,[n-2]}-F_{n,[n-1]}=F_{n,C}-F_{n,C\cup\{n-1\}}=\mathbf{f}_{C,k}.
\end{align*}
We have now shown that $(\mathbf{m}_{C,k})_{(C,k)\in\Omega_{n}}$ expands invertibly triangularly in $(\mathbf{f}_{C,k})_{(C,k)\in\Omega_{n}}$, so the result follows from Lemma \ref{l-invtri}.
\end{proof}

We are now ready to give the proof of Theorem \ref{t-M} (a).

\begin{proof}[Proof of Theorem \ref{t-M} \textup{(}a\textup{)}]
We have
\begin{align*}
\mathcal{K}_{n}^{\Pk} & =\Span\mathcal{F}_{n}^{\Pk} & \text{(by Theorem \ref{t-F} (a))}\\
 & =\Span(\mathbf{f}_{C,k})_{(C,k)\in\Omega_{n}} & \text{(by Proposition \ref{p-PkftoF})}\\
 & =\Span(\mathbf{m}_{C,k})_{(C,k)\in\Omega_{n}} & \text{(by Proposition \ref{p-Pkftom})}\\
 & =\Span\mathcal{M}_{n}^{\Pk} & \text{(by Proposition \ref{p-PkmtoM})}
\end{align*}
for all $n\geq0$. Therefore,
\begin{align*}
\mathcal{K}^{\Pk} & =\bigoplus_{n=0}^{\infty}\mathcal{K}_{n}^{\Pk}=\bigoplus_{n=0}^{\infty}\Span\mathcal{M}_{n}^{\Pk}\\
 & =\Span\left(\{\,M_{J}+M_{K}:J\triangleright_{1}K\text{ or }J\triangleright_{2}K\,\}\cup\{\,M_{J}:J\in\mathcal{\tilde{C}}\,\}\right).\qedhere
\end{align*}
\end{proof}

\subsection{\label{ss-Mpk}The peak number\textemdash proof of Theorem \ref{t-M} (b)}

Let us now proceed to $\mathcal{K}^{\pk}$. Define 
\begin{align*}
\mathcal{F}_{n}^{\pk} & \coloneqq\{\,F_{J}-F_{K}:J\rightarrow_{\pk}K\text{ and }J,K\vDash n\,\}\quad\text{and}\\
\mathcal{M}_{n}^{\pk} & \coloneqq\{M_{J}+M_{K}:J\triangleright_{1}K\text{ or }J\triangleright_{2}K,\text{ and }J,K\vDash n\}\cup\{M_{(1^{n-2},2)}\}\\
 & \qquad\qquad\qquad\qquad\qquad\qquad\qquad\qquad\cup\{\,M_{J}-M_{K}:J\rightarrow_{3}K\text{ and }J,K\vDash n\,\},
\end{align*}
so that we wish to prove $\mathcal{K}_{n}^{\pk}=\Span\mathcal{M}_{n}^{\pk}$ for all $n$. 

Recall the definitions of $\Omega_{n,1}$, $\Omega_{n,2}$, and $\Omega_{n,3}$ from Section \ref{ss-MPk}, and define $\Omega_{n,4}\subseteq2^{[n-1]}\times[n-1]$ by
\begin{alignat*}{1}
\Omega_{n,4} & \coloneqq\big\{\,(C,k):n-1\in C,\:k-1\in C\cup\{0\},\:k\in C,\:k+1\notin C,\:k+2\in C,\\
 & \qquad\qquad\qquad\text{and }j\in(C\cup\{0\})\backslash\{n-1\}\implies j+1\in C\text{ or }j+2\in C\,\big\}.
\end{alignat*}
Note that $\Omega_{n,4}$ is empty for $n\leq3$, and that $\Omega_{n,4}$ is disjoint from $\Omega_{n,1}$, $\Omega_{n,2}$, and $\Omega_{n,3}$. Write 
\[
\Theta_{n}\coloneqq\Omega_{n,1}\sqcup\Omega_{n,2}\sqcup\Omega_{n,3}\sqcup\Omega_{n,4}=\Omega_{n}\sqcup\Omega_{n,4}.
\]

Next, we shall expand the definitions of $\mathbf{f}_{C,k}$ and $\mathbf{m}_{C,k}$ from Section \ref{ss-MPk} to all $(C,k)\in\Theta_{n}$. Let
\[
\mathbf{f}_{C,k}\coloneqq\begin{cases}
F_{n,C}-F_{n,C\cup\{k-1\}}, & \text{if }(C,k)\in\Omega_{n,1},\\
F_{n,C}-F_{n,C\cup\{n-1\}}, & \text{if }(C,k)\in\Omega_{n,2}\text{ or }(C,k)\in\Omega_{n,3},\\
F_{n,C}-F_{n,(C\cup\{k+1\})\backslash\{k\}}, & \text{if }(C,k)\in\Omega_{n,4},
\end{cases}
\]
and 
\[
\mathbf{m}_{C,k}\coloneqq\begin{cases}
M_{n,C}+M_{n,C\cup\{k\}}, & \text{if }(C,k)\in\Omega_{n,1}\text{ or }(C,k)\in\Omega_{n,2},\\
M_{n,C}, & \text{if }(C,k)\in\Omega_{n,3},\\
M_{n,C}-M_{n,(C\cup\{k+1\})\backslash\{k\}}, & \text{if }(C,k)\in\Omega_{n,4}.
\end{cases}
\]
The proof of the next lemma is routine and so it is omitted.

\begin{lem} \label{l-Om4}
For all $n\geq0$ and $J,K\vDash n$, we have $J\rightarrow_{3}K$ if and only if $\Des K=(\Des J\cup\{k+1\})\backslash\{k\}$ and $(\Des J,k)\in\Omega_{n,4}$ for some $k\in [n-1]$.
\end{lem}

\begin{prop} \label{p-pkftoF}
For all $n\geq0$, we have $\Span(\mathbf{f}_{C,k})_{(C,k)\in\Theta_{n}}=\Span\mathcal{F}_{n}^{\pk}$.
\end{prop}

\begin{proof}
In light of Proposition \ref{p-PkftoF}, it suffices to show that
\begin{equation}
\Span(\mathbf{f}_{C,k})_{(C,k)\in\Omega_{n,4}}=\Span\{\,F_{J}-F_{K}:J\rightarrow_{3}K\text{ and }J,K\vDash n\,\}.\label{e-spansfF}
\end{equation}
Fix $(C,k)\in\Omega_{n,4}$. Let $J=\Comp C$ and $K=\Comp((C\cup\{k+1\})\backslash\{k\})$, so that $\Des K=(\Des J\cup\{k+1\})\backslash\{k\}$ and $(\Des J,k)\in\Omega_{n,4}$. By Lemma \ref{l-Om4}, we have
\begin{align*}
\mathbf{f}_{C,k} & =F_{n,C}-F_{n,(C\cup\{k+1\})\backslash\{k\}}=F_{J}-F_{K}
\end{align*}
where $J\rightarrow_{3}K$, so the forward inclusion of (\ref{e-spansfF}) follows from linearity. The reverse inclusion is similar.
\end{proof}

\begin{prop} \label{p-pkmtoM}
For all $n\geq0$, we have $\Span(\mathbf{m}_{C,k})_{(C,k)\in\Theta_{n}}=\Span\mathcal{M}_{n}^{\pk}$.
\end{prop}

\begin{proof}
Given Proposition \ref{p-PkmtoM}, it suffices to show that 
\[
\Span(\mathbf{m}_{C,k})_{(C,k)\in\Omega_{n,4}}=\Span\{\,M_{J}-M_{K}:J\rightarrow_{3}K\text{ and }J,K\vDash n\,\},
\]
but this is proved in the same way as (\ref{e-spansfF}) and so we omit the details.
\end{proof}

\begin{prop} \label{p-pkftom}
For all $n\geq0$, we have $\Span(\mathbf{f}_{C,k})_{(C,k)\in\Theta_{n}}=\Span(\mathbf{m}_{C,k})_{(C,k)\in\Theta_{n}}$.
\end{prop}

\begin{proof}
Take the partial order defined on each $\Omega_{n,i}$ from the proof of Proposition \ref{p-Pkftom} and extend it to $\Omega_{n,4}$. Then, endow $\Theta_{n}$ with the partial order obtained by taking the disjoint union of the posets $\Omega_{n}$ and $\Omega_{n,4}$, so that the elements of $\Omega_{n}$ are incomparable with those in $\Omega_{n,4}$. The proof of Proposition \ref{p-Pkftom} established that $\mathbf{m}_{C,k}$ for each $(C,k)\in\Omega_{n}$ has an invertibly triangular expansion in $(\mathbf{f}_{C,k})_{(C,k)\in\Omega_{n}}$ and thus $(\mathbf{f}_{C,k})_{(C,k)\in\Theta_{n}}$; we shall show that $\mathbf{m}_{C,k}$ for each $(C,k)\in\Omega_{n,4}$ has such an expansion as well.

Fix $(C,k)\in\Omega_{n,4}$. Then
\begin{align}
\mathbf{m}_{C,k} & =M_{n,C}-M_{n,(C\cup\{k+1\})\backslash\{k\}};\label{e-M1}
\end{align}
we shall expand $M_{n,C}$ and $M_{n,(C\cup\{k+1\})\backslash\{k\}}$ separately. From Lemma \ref{l-MtoF} (a), we have 
\begin{align}
M_{n,C} & =\sum_{C\subseteq B\subseteq[n-1]}(-1)^{\left|B\backslash C\right|}F_{n,B}\nonumber \\
 & =\sum_{\substack{C\subseteq B\subseteq[n-1]\\
k+1\notin B
}
}(-1)^{\left|B\backslash C\right|}F_{n,B}+\sum_{\substack{C\subseteq B\subseteq[n-1]\\
k+1\in B
}
}(-1)^{\left|B\backslash C\right|}F_{n,B}\label{e-M2}
\end{align}
and 
\begin{align*}
M_{n,(C\cup\{k+1\})\backslash\{k\}} & =\sum_{(C\cup\{k+1\})\backslash\{k\}\subseteq B\subseteq[n-1]}(-1)^{\left|B\backslash((C\cup\{k+1\})\backslash\{k\})\right|}F_{n,B}\\
 & =\sum_{\substack{C\backslash\{k\}\subseteq B\subseteq[n-1]\\
k+1\in B
}
}(-1)^{\left|B\backslash((C\cup\{k+1\})\backslash\{k\})\right|}F_{n,B}.
\end{align*}
Note that $(C,k)\in\Omega_{n,4}$ implies $k+1\notin C$, so if $k+1\in B$ then $\left|B\backslash((C\cup\{k+1\})\backslash\{k\})\right|=\left|B\backslash(C\backslash\{k\})\right|-1$. Hence, 
\begin{align*}
M_{n,(C\cup\{k+1\})\backslash\{k\}} & =\sum_{\substack{C\backslash\{k\}\subseteq B\subseteq[n-1]\\
k+1\in B
}
}(-1)^{\left|B\backslash(C\backslash\{k\})\right|-1}F_{n,B}\\
 & =-\sum_{\substack{C\backslash\{k\}\subseteq B\subseteq[n-1]\\
k+1\in B
}
}(-1)^{\left|B\backslash(C\backslash\{k\})\right|}F_{n,B}\\
 & =-\sum_{\substack{C\backslash\{k\}\subseteq B\subseteq[n-1]\\
k,k+1\in B
}
}(-1)^{\left|B\backslash(C\backslash\{k\})\right|}F_{n,B}-\sum_{\substack{C\backslash\{k\}\subseteq B\subseteq[n-1]\\
k\notin B,\,k+1\in B
}
}(-1)^{\left|B\backslash(C\backslash\{k\})\right|}F_{n,B}.
\end{align*}
Furthermore, $(C,k)\in\Omega_{n,4}$ implies $k\in C$, so the two conditions $C\backslash\{k\}\subseteq B$ and $k\in B$ are equivalent to the single condition $C\subseteq B$. Also, if $k\in B$ then $\left|B\backslash(C\backslash\{k\})\right|=1+ \left|B\backslash C\right|$, whereas if $k\notin B$ then $\left|B\backslash(C\backslash\{k\})\right|=\left|B\backslash C\right|$.
Therefore, 
\begin{align}
M_{n,(C\cup\{k+1\})\backslash\{k\}} & =\sum_{\substack{C\subseteq B\subseteq[n-1]\\
k+1\in B
}
}(-1)^{\left|B\backslash C\right|}F_{n,B}-\sum_{\substack{C\backslash\{k\}\subseteq B\subseteq[n-1]\\
k\notin B,\,k+1\in B
}
}(-1)^{\left|B\backslash C\right|}F_{n,B}.\label{e-M3}
\end{align}
Substituting (\ref{e-M2}) and (\ref{e-M3}) into (\ref{e-M1}) yields
\begin{align*}
\mathbf{m}_{C,k} & =\sum_{\substack{C\subseteq B\subseteq[n-1]\\
k+1\notin B
}
}(-1)^{\left|B\backslash C\right|}F_{n,B}+\sum_{\substack{C\backslash\{k\}\subseteq B\subseteq[n-1]\\
k\notin B,\,k+1\in B
}
}(-1)^{\left|B\backslash C\right|}F_{n,B}\\
 & =\sum_{\substack{C\backslash\{k\}\subseteq B\subseteq[n-1]\\
k\in B,\,k+1\notin B
}
}(-1)^{\left|B\backslash C\right|}F_{n,B}+\sum_{\substack{C\backslash\{k\}\subseteq B\subseteq[n-1]\\
k\notin B,\,k+1\in B
}
}(-1)^{\left|B\backslash C\right|}F_{n,B}.
\end{align*}
It is straightforward to verify that the map $B\mapsto(B\cup\{k+1\})\backslash\{k\}$ is a bijection from sets $B$ satisfying $C\backslash\{k\}\subseteq B\subseteq[n-1]$, $k\in B$, and $k+1\notin B$ to those satisfying $C\backslash\{k\}\subseteq B\subseteq[n-1]$, $k\notin B$, and $k+1\in B$. Moreover, for the former family of sets, we have $\left|((B\cup\{k+1\})\backslash\{k\})\backslash C\right|=1+\left|(B\backslash\{k\})\backslash C\right|=1+\left|B\backslash C\right|$. Therefore, we have 
\begin{align*}
\mathbf{m}_{C,k} & =\sum_{\substack{C\backslash\{k\}\subseteq B\subseteq[n-1]\\
k\in B,\,k+1\notin B
}
}\Big((-1)^{\left|B\backslash C\right|}F_{n,B}+(-1)^{\left|((B\cup\{k+1\})\backslash\{k\})\backslash C\right|}F_{n,(B\cup\{k+1\})\backslash\{k\}}\Big)\\
 & =\sum_{\substack{C\backslash\{k\}\subseteq B\subseteq[n-1]\\
k\in B,\,k+1\notin B
}
}(-1)^{\left|B\backslash C\right|}(F_{n,B}-F_{n,(B\cup\{k+1\})\backslash\{k\}}).\\
 & =\sum_{\substack{C\subseteq B\subseteq[n-1]\\
k+1\notin B
}
}(-1)^{\left|B\backslash C\right|}(F_{n,B}-F_{n,(B\cup\{k+1\})\backslash\{k\}}).
\end{align*}
Finally, it is readily checked that 
\[
\{\,(B,k):C\subseteq B\subseteq[n-1],k+1\notin B\,\}=\{\,(B,l)\in\Omega_{n,4}:(B,l)\geq(C,k)\,\},
\]
whence
\begin{align*}
\mathbf{m}_{C,k} & =\sum_{\substack{(B,l)\in\Omega_{n,4}\\
(B,l)\geq(C,k)
}
}(-1)^{\left|B\backslash C\right|}(F_{n,B}-F_{n,(B\cup\{l+1\})\backslash\{l\}})\\
 & =\sum_{\substack{(B,l)\in\Omega_{n,4}\\
(B,l)\geq(C,k)
}
}(-1)^{\left|B\backslash C\right|}\mathbf{f}_{B,l}\\
 & =\mathbf{f}_{C,k}+\sum_{\substack{(B,l)\in\Theta_{n}\\
(B,l)>(C,k)
}
}(-1)^{\left|B\backslash C\right|}\mathbf{f}_{B,l}.
\end{align*}

We have shown that $(\mathbf{m}_{C,k})_{(C,k)\in\Theta_{n}}$ expands invertibly triangularly in $(\mathbf{f}_{C,k})_{(C,k)\in\Theta_{n}}$, so the desired conclusion is reached via Lemma \ref{l-invtri}.
\end{proof}

Finally, we complete the proof of Theorem \ref{t-M} (b) using the preceding propositions.

\begin{proof}[Proof of Theorem \ref{t-M} \textup{(}b\textup{)}]
We have 
\begin{align*}
\mathcal{K}_{n}^{\pk} & =\Span\mathcal{F}_{n}^{\pk} & \text{(by Theorem \ref{t-F} (b))}\\
 & =\Span(\mathbf{f}_{C,k})_{(C,k)\in\Theta_{n}} & \text{(by Proposition \ref{p-pkftoF})}\\
 & =\Span(\mathbf{m}_{C,k})_{(C,k)\in\Theta_{n}} & \text{(by Proposition \ref{p-pkftom})}\\
 & =\Span\mathcal{M}_{n}^{\pk} & \text{(by Proposition \ref{p-pkmtoM})}
\end{align*}
for all $n\geq0$, and therefore
\begin{align*}
\mathcal{K}^{\pk} & =\bigoplus_{n=0}^{\infty}\mathcal{K}_{n}^{\pk}=\bigoplus_{n=0}^{\infty}\Span\mathcal{M}_{n}^{\pk}\\
 & =\Span\!\Big(\{\,M_{J}+M_{K}:J\triangleright_{1}K\text{ or }J\triangleright_{2}K\,\}\cup\{\,M_{J}:J\in\mathcal{\tilde{C}}\,\}\cup\{\,M_{J}-M_{K}:J\rightarrow_{3}K\,\}\Big)
\end{align*}
as desired.
\end{proof}

\section{\label{ss-val}Valleys}

Given $\pi\in\mathfrak{P}_{n}$, we say that $i\in\{2,3,\dots,n-1\}$ is a\textit{ valley} of $\pi$ if $\pi_{i-1}>\pi_{i}<\pi_{i+1}$. Then $\Val\pi$ is defined to be the set of valleys of $\pi$ and $\val\pi$ the number of valleys of $\pi$. There is a clear symmetry relating peaks and valleys in permutations, and this section will describe the implications that this symmetry has on the statistics $\Pk$, $\pk$, $\Val$, and $\val$, their shuffle algebras, and their kernels. In particular, we will obtain from Theorem \ref{t-F} an analogue of this theorem for $\mathcal{K}^{\Val}$ and $\mathcal{K}^{\val}$. 

The \textit{complement} $\pi^{c}$ of $\pi\in\mathfrak{P}_{n}$ is the permutation obtained by (simultaneously) replacing the $i$th smallest letter in $\pi$ with the $i$th largest letter in $\pi$ for all $1\leq i\leq n$. For example, if $\pi=472691$ then $\pi^{c}=627419$. Observe that
\[
\Pk\pi=\Val\pi^{c}\quad\text{and}\quad\pk\pi=\val\pi^{c};
\]
this implies that $\Pk$ and $\Val$ are \textit{$c$-equivalent statistics} (see \cite[Section 3.2]{Gessel2018} for the definition), and so are $\pk$ and $\val$. Complementation is a ``shuffle-compatibility-preserving'' involution on permutations, and according to \cite[Theorem 3.5]{Gessel2018}, if two permutation statistics $\st_{1}$ and $\st_{2}$ are $f$-equivalent where $f$ is a shuffle-compatibility-preserving involution and $\st_{1}$ is shuffle-compatible, then $\st_{2}$ is also shuffle-compatible and the map $[\pi]_{\st_{1}}\mapsto[\pi^{f}]_{\st_{2}}$ extends to an isomorphism between their shuffle algebras. Consequently, we have $\mathcal{A}^{\Pk}\cong\mathcal{A}^{\Val}$ and $\mathcal{A}^{\pk}\cong\mathcal{A}^{\val}$.
\begin{figure}
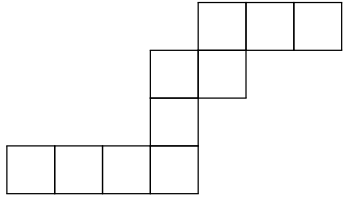

\noindent \begin{centering}
\ydiagram{4+3,3+2,3+1,4}
\par\end{centering}
\caption{\label{f-rshape}The ribbon shape of the composition $(4,1,2,3)$.}
\end{figure}

We can also define complements of compositions in the following way. Given $L\vDash n$, let $L^{c}\coloneqq\Comp([n-1]\backslash\Des L)$. For example, if $L=(4,1,2,3)$ then we have $\Des L=\{4,5,7\}$ and thus $L^{c}=\Comp\{1,2,3,6,8,9\}=(1,1,1,3,2,1,1)$. A simple way to obtain $L^{c}$ from $L$ is to draw its \textit{ribbon shape} (see Figure
\ref{f-rshape} for an example); reading off the columns from left to right (as opposed to the rows from bottom to top) yields $L^{c}$. It is evident that $\Comp\pi=L$ implies $\Comp\pi^{c}=L^{c}$, so we have 
\[
\Pk L=\Val L^{c}\quad\text{and}\quad\pk L=\val L^{c}.
\]
Hence, two compositions are $\Pk$-equivalent (respectively, $\pk$-equivalent) if and only if their complements are $\Val$-equivalent (respectively, $\val$-equivalent). 

Consider the involutory automorphism $\psi$ on $\QSym$ defined by $\psi(F_{L})=F_{L^{c}}$ \cite[Section 3.6]{Luoto2013}. Then the following diagrams commute:
\vspace{-10bp}
\begin{figure}[H]
\noindent \begin{centering}
\begin{center}
\begin{tikzpicture}[scale=0.4,auto,every edge quotes/.style = {sloped}]

\node (3) at (0,24) {$\QSym$};
\node (4) at (12,24) {$\QSym$};

\node (1) at (0,12) {$\mathcal{A}^{\Pk}$};
\node (2) at (12,12) {$\mathcal{A}^{\Val}$};

\node (5) at (0,0) {$\mathcal{A}^{\pk}$};
\node (6) at (12,0) {$\mathcal{A}^{\val}$};

\draw[<->] (1) edge["${[\pi]}_{\Pk} \mapsto {[\pi^c]}_{\Val}$"] (2);
\draw[<->] (3) edge["$\psi$"] (4);
\draw[<->] (5) edge["${[\pi]}_{\pk} \mapsto {[\pi^c]}_{\val}$"] (6);

\draw[->>] (1) edge["${[\pi]}_{\Pk} \mapsto {[\pi]}_{\pk}$"] (5);
\draw[->>] (2) edge["${[\pi]}_{\Val} \mapsto {[\pi]}_{\val}$"] (6);

\draw[->>] (3) edge["$p_{\Pk}$"] (1);
\draw[->>] (4) edge["$p_{\Val}$"] (2);
\draw[->>] (3) edge["$p_{\pk}$", bend right] (5);
\draw[->>] (4) edge["$p_{\val}$", bend left] (6);

\node (7) at (22,20) {$\QSym$};
\node (8) at (30,20) {$\QSym$};

\node (9) at (22,12) {$\mathcal{K}^{\pk}$};
\node (10) at (30,12) {$\mathcal{K}^{\val}$};

\node (11) at (22,4) {$\mathcal{K}^{\Pk}$};
\node (12) at (30,4) {$\mathcal{K}^{\Val}$};

\draw[<->] (7) edge["$\psi$"] (8);
\draw[<->] (9) edge["$\left.\psi\right|_{\mathcal{K}^{\pk}}$"] (10);
\draw[<->] (11) edge["$\left.\psi\right|_{\mathcal{K}^{\Pk}}$"] (12);

\path[right hook->] (9) edge (7);
\path[right hook->] (10) edge (8);
\path[right hook->] (11) edge (9);
\path[right hook->] (12) edge (10);

\end{tikzpicture}
\end{center}\smallskip{}
\par\end{centering}
\caption{\label{f-comdiag}Relationships between the shuffle algebras and kernels of $\protect\Pk$, $\protect\pk$, $\protect\Val$, and $\protect\val$}
\end{figure}
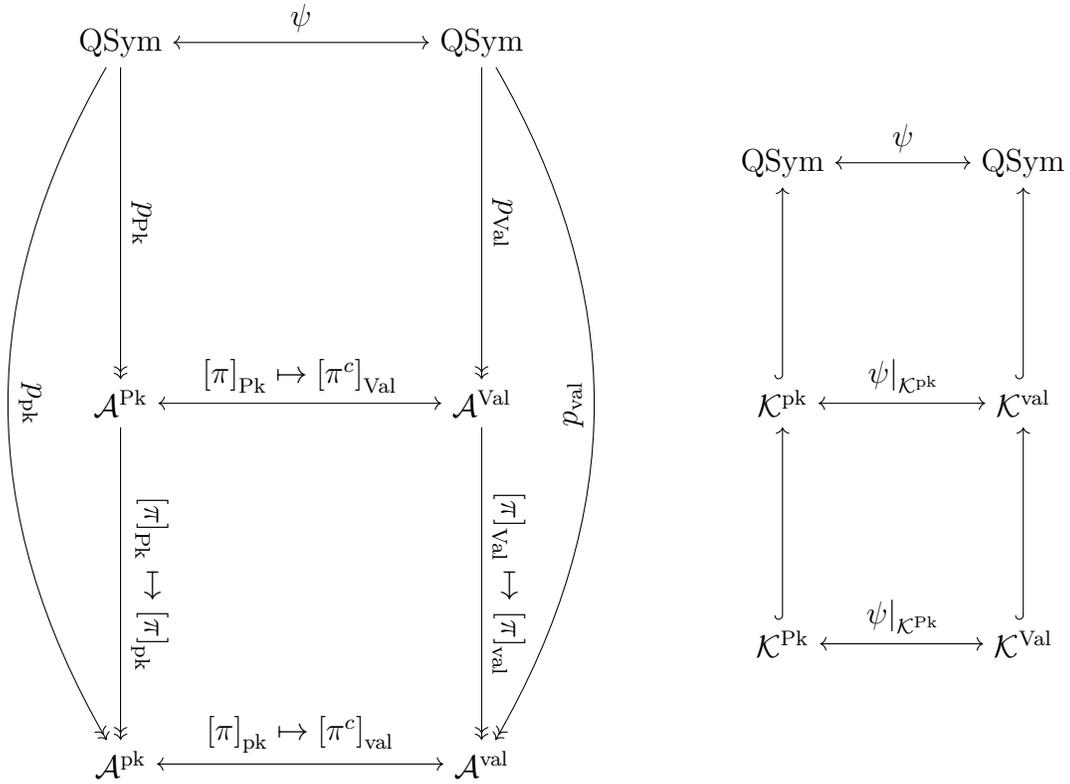

\noindent The involution $\psi$ allows us to obtain characterizations for the ideals $\mathcal{K}^{\Val}$ and $\mathcal{K}^{\val}$ in terms of the fundamental basis from the analogous results for $\mathcal{K}^{\Pk}$ and $\mathcal{K}^{\pk}$ (Theorem~\ref{t-F}). Given compositions $J=(j_{1},j_{2},\dots,j_{m})$ and $K$, let us write: 
\begin{itemize}
\item $J\twoheadrightarrow_{1}K$ if there exists $l\in[m-1]$ for which
$j_{l}\geq2$ and $j_{l+1}=1$, and 
\[
K=(j_{1},\dots,j_{l-1},j_{l}+1,j_{l+2},\dots,j_{m});
\]
\item $J\twoheadrightarrow_{2}K$ if $j_{1}=j_{2}=1$ and
\[
K=(2,j_{3},\dots,j_{m});
\]
\item $J\twoheadrightarrow_{3}K$ if $j_{i}\geq2$ for all $i\in\{2,3,\dots,m\}$, there exists $l\in[m-1]$ for which $j_{l}\geq2$---and $j_{l}>2$ if $l>1$---and
\[
K=(j_{1},\dots,j_{l-1},j_{l}-1,j_{l+1}+1,j_{l+2},\dots,j_{m}).
\]
\end{itemize}

\begin{thm} \label{t-vals}
The ideals $\mathcal{K}^{\Val}$ and $\mathcal{K}^{\val}$ are spanned \textup{(}as $\mathbb{Q}$-vector spaces\textup{)} in the following ways\textup{:}
\begin{enumerate}
\item [\normalfont{(a)}]$\mathcal{K}^{\Val}=\Span\{\,F_{J}-F_{K}:J\twoheadrightarrow_{1}K\text{ or }J\twoheadrightarrow_{2}K\,\}$
\item [\normalfont{(b)}]$\mathcal{K}^{\val}=\Span\{\,F_{J}-F_{K}:J\twoheadrightarrow_{1}K, \ J\twoheadrightarrow_{2}K,\text{ or }J\twoheadrightarrow_{3}K\,\}$
\end{enumerate}
\end{thm}

To prove Theorem \ref{t-vals}, we will need a couple lemmas.
\begin{lem}
\label{l-pkcomp}Let $J$ and $K$ be compositions.
\begin{enumerate}
\item [\normalfont{(a)}]If $J\rightarrow_{\Pk}K$, then $J^{c}\twoheadrightarrow_{1}K^{c}$
or $J^{c}\twoheadrightarrow_{2}K^{c}$.
\item [\normalfont{(b)}]If $J\rightarrow_{\pk}K$, then $J^{c}\twoheadrightarrow_{1}K^{c}$,
$J^{c}\twoheadrightarrow_{2}K^{c}$, or $J^{c}\twoheadrightarrow_{3}K^{c}$.
\end{enumerate}
\end{lem}

\begin{proof}
Let $J=(j_{1},j_{2},\dots,j_{m})$ and suppose that $J\rightarrow_{\Pk}K$,
so either $J\rightarrow_{1}K$ or $J\rightarrow_{2}K$. 
\begin{enumerate}
\item [\normalfont{(1)}]Suppose that $J\rightarrow_{1}K$. Then there exists
$l\in[m]$ for which $j_{l}>2$ and 
\[
K=(j_{1},\dots,j_{l-1},1,j_{l}-1,j_{l+1},\dots,j_{m}).
\]
Upon drawing the ribbon diagrams of $J$ and $K$ and reading the columns from left to right, we see that $J^{c}$ and $K^{c}$ are the same except that a segment $(\alpha,1^{j_{l}-2})$ of $J^{c}$ is replaced with $(\alpha+1,1^{j_{l}-3})$ in $K^{c}$:

\begin{figure}[H]
\begin{center}
\begin{tikzpicture}[]

\draw (0,0) node {
\begin{ytableau} 
\none & \none & \none & \vdots & \none & \none & \none & \none & \none & \none & \none & \none & \vdots \\ 
\phantom{.} & \phantom{.} & \cdots & \phantom{.} & \none & \none & \none & \none & \none & \none & \phantom{.} & \cdots & \phantom{.} \\ 
\vdots & \none & \none & \none & \none & \none & \none & \none & \none & \none & \phantom{.} \\
\none & \none & \none & \none & \none & \none & \none & \none & \none & \none & \vdots
\end{ytableau}
};

\draw [thick, decorate,decoration={brace,amplitude=5pt},yshift=2pt] (-3.45,0.72) -- (-2.28,0.72) node [black,midway,yshift=12pt, scale=0.75] {$j_l -2$};

\draw [thick, decorate,decoration={brace,amplitude=5pt},yshift=2pt] (-4.3,-0.69) -- (-4.3,0.55) node [black,midway,xshift=-12pt, scale=0.75] {$\alpha$};

\draw [thick, decorate,decoration={brace,amplitude=5pt},yshift=2pt] (2.91,0.72) -- (3.42,0.72) node [black,midway,yshift=12pt, scale=0.75] {$j_l -3 \quad$};

\draw [thick, decorate,decoration={brace,amplitude=5pt},yshift=2pt] (2.05,-1.32) -- (2.05,0.55) node [black,midway,xshift=-20pt, scale=0.75] {$\alpha+1$};

\draw (0,0.25) node[black, scale=1.5] {$\mapsto$};

\end{tikzpicture}
\end{center}\vspace{-15bp}
\caption{Ribbon diagram upon replacing a part $j_{l}>2$ with $(1,j_{l}-1)$.}
\end{figure}
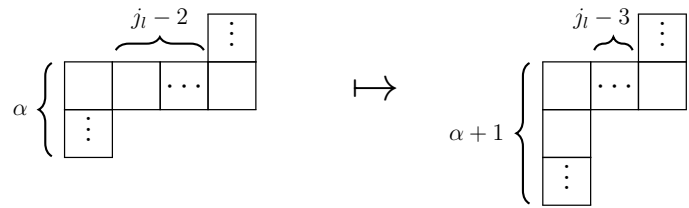

If $l=1$, then $\alpha=1$ and we have $J^{c}\twoheadrightarrow_{2}K^{c}$. Otherwise, if $l>1$, then $\alpha\geq2$ and we have $J^{c}\twoheadrightarrow_{1}K^{c}$.
\item [\normalfont{(2)}]Suppose that $J\rightarrow_{2}K$, so that $j_{m}=2$
and 
\[
K=(j_{1},\dots,j_{m-1},1,1).
\]
Then $J^{c}$ and $K^{c}$ are the same except that $J^{c}$ ends with $(\alpha,1)$ and $K^{c}$ ends with $\alpha+1$:

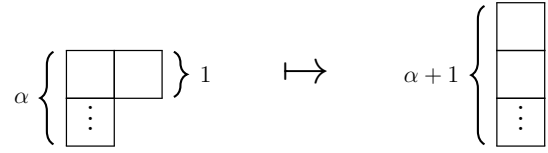
\begin{figure}[H]
\begin{center}
\begin{tikzpicture}[]

\draw (0,0) node {
\begin{ytableau} 
\none  & \none  & \none & \none & \none & \none & \none & \none & \none & \phantom{.} \\
\phantom{.} & \phantom{.} & \none & \none & \none & \none & \none & \none & \none & \phantom{.} \\ 
\vdots & \none & \none & \none & \none & \none & \none & \none & \none & \vdots
\end{ytableau}
};

\draw [thick, decorate,decoration={brace,amplitude=5pt},yshift=2pt] (-3.35,-1) -- (-3.35,0.23) node [black,midway,xshift=-12pt, scale=0.75] {$\alpha$};

\draw [thick, decorate,decoration={brace,amplitude=5pt, mirror},yshift=2pt] (-1.75,-0.37) -- (-1.75,0.23) node [black,midway,xshift=12pt, scale=0.75] {$1$};

\draw [thick, decorate,decoration={brace,amplitude=5pt},yshift=2pt] (2.38,-1) -- (2.38,0.84) node [black,midway,xshift=-20pt, scale=0.75] {$\alpha+1$};

\draw (0,0) node[black, scale=1.5] {$\mapsto$};

\end{tikzpicture}
\end{center}\vspace{-15bp}
\caption{Ribbon diagram upon replacing the final part $j_{m}=2$ with $(1,1)$.}
\end{figure}

If $\alpha\geq2$, then $J^{c}\twoheadrightarrow_{1}K^{c}$. Otherwise, $\alpha=1$ forces $m=1$ and thus we have $J^{c}\twoheadrightarrow_{2}K^{c}$.
\end{enumerate}
Part (a) follows from the two cases above.

In light of (a), to prove (b) it suffices to show that $J\rightarrow_{3}K$ implies $J^{c}\twoheadrightarrow_{3}K^{c}$. To that end, let us suppose that $J\rightarrow_{3}K$, so that $j_{i}\leq2$ for all $i\in[m]$, $j_{m}=1$, $j_{l}=1$ and $j_{l+1}=2$ for some $l\in[m-2]$, and 
\[
K=(j_{1},\dots,j_{l-1},j_{l+1},j_{l},j_{l+2},\dots,j_{m}).
\]
Then $J^{c}$ and $K^{c}$ are the same except that a segment $(\alpha,\beta)$ of $J^{c}$ is replaced with $(\alpha-1,\beta+1)$ in $K^{c}$, where $\alpha\geq2$ if $\alpha$ is the first part of $J^{c}$ and $\alpha\geq3$ otherwise:
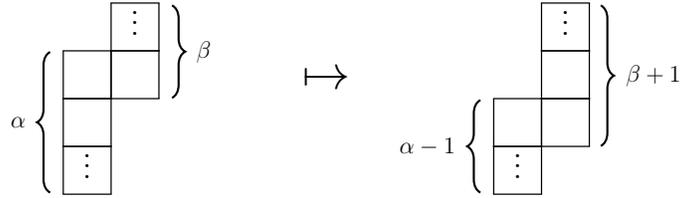
\begin{figure}[H]
\begin{center}
\begin{tikzpicture}[]

\draw (0,0) node {
\begin{ytableau} 
\none & \vdots & \none & \none & \none & \none & \none & \none & \none & \none & \vdots \\
\phantom{.} & \phantom{.} & \none & \none & \none & \none & \none & \none & \none & \none & \phantom{.} \\
\phantom{.} & \none & \none & \none & \none & \none & \none & \none & \none & \phantom{.} & \phantom{.} \\
\vdots & \none & \none & \none & \none & \none & \none & \none & \none & \vdots
\end{ytableau}
};

\draw [thick, decorate,decoration={brace,amplitude=5pt},yshift=2pt] (-3.67,-1.32) -- (-3.67,0.55) node [black,midway,xshift=-12pt, scale=0.75] {$\alpha$};

\draw [thick, decorate,decoration={brace,amplitude=5pt, mirror},yshift=2pt] (-2.05,-0.06) -- (-2.05,1.17) node [black,midway,xshift=12pt, scale=0.75] {$\beta$};

\draw [thick, decorate,decoration={brace,amplitude=5pt},yshift=2pt] (2.05,-1.32) -- (2.05,-0.09) node [black,midway,xshift=-20pt, scale=0.75] {$\alpha-1$};

\draw [thick, decorate,decoration={brace,amplitude=5pt, mirror},yshift=2pt] (3.65,-0.7) -- (3.65,1.17) node [black,midway,xshift=20pt, scale=0.75] {$\beta+1$};

\draw (0,0.25) node[black, scale=1.5] {$\mapsto$};

\end{tikzpicture}
\end{center}\vspace{-15bp}
\caption{Ribbon diagram upon swapping the positions of parts $(1,2)$.}
\end{figure}
\noindent Furthermore, requiring $J$ to have every part at most 2 and for $J$ to end with a 1 implies that every part of $J^{c}$, except possibly the first one, is at least 2. Thus $J^{c}\twoheadrightarrow_{3}K^{c}$, and the proof is complete.
\end{proof}

Just as the peaks of a permutation occur precisely at the end of its non-final increasing runs of length at least 2, the valleys occur precisely at the beginning of its non-initial increasing runs of length at least 2. Then the following lemma readily follows; compare with Lemma \ref{l-PkpkL} for peaks.
\begin{lem}
\label{l-ValvalL}Let $L=(j_{1},j_{2},\dots,j_{m})$ be a composition. Then:
\begin{enumerate}
\item [\normalfont{(a)}]$\Val L=\left\{ \,1+\sum_{i=1}^{k-1}j_{i}:j_{k}\geq2\text{ and }k\in\{2,3,\dots,m\}\,\right\} $
\item [\normalfont{(b)}]$\val L=\left|\{\,k\in\{2,3,\dots,m\}:j_{k}\geq2 \,\}\right|$
\end{enumerate}
\end{lem}

Now we are ready to prove Theorem \ref{t-vals}.
\begin{proof}[Proof of Theorem \ref{t-vals}]
In light of Theorem \ref{t-F} and the involution $\psi$, Lemma \ref{l-pkcomp} implies the forward inclusions 
\begin{align*}
\mathcal{K}^{\Val} & \subseteq\Span\{\,F_{J}-F_{K}:J\twoheadrightarrow_{1}K\text{ or }J\twoheadrightarrow_{2}K\,\}\quad\text{and}\\
\mathcal{K}^{\val} & \subseteq\Span\{\,F_{J}-F_{K}:J\twoheadrightarrow_{1}K,\ J\twoheadrightarrow_{2}K,\text{ or }J\twoheadrightarrow_{3}K\,\}.
\end{align*}
For the reverse inclusions, it suffices to show that $J\twoheadrightarrow_{1}K$ and $J\twoheadrightarrow_{2}K$ each imply $J\sim_{\Val}K$, and that $J\twoheadrightarrow_{1}K$, $J\twoheadrightarrow_{2}K$, and
$J\twoheadrightarrow_{3}K$ each imply $J\sim_{\val}K$. However, it is readily seen from Lemma \ref{l-ValvalL} and the definitions of $\twoheadrightarrow_{1}$, $\twoheadrightarrow_{2}$, and $\twoheadrightarrow_{3}$ that these indeed hold.
\end{proof}

Ehrenborg's formula
\[
\psi(M_{L})=(-1)^{n-\ell(L)}\sum_{L\leq K}M_{K}
\]
\cite[Section 5]{Ehrenborg1996}, where $\ell(L)$ is the number of parts of $L$, gives us the image of the involution $\psi$ on the monomial basis. One can use this formula along with Theorem \ref{t-M} to obtain spanning sets for $\mathcal{K}^{\Val}$ and $\mathcal{K}^{\val}$ in terms of monomial quasisymmetric functions, but more work is needed to show that $\Val$ and $\val$ are $M$-binomial. We will not do this here.

Finally, it is worth mentioning that the kernel $\mathcal{K}^{\epk}$ of the exterior peak number is identical to $\mathcal{K}^{\val}$. This is because $\epk\pi=\val\pi+1$ for all (nonempty) permutations $\pi$ \cite[Lemma 2.1 (e)]{Gessel2018}, and therefore $J\sim_{\epk}K$ if and only if $J\sim_{\val}K$. 

\section{\label{ss-future}Future directions of research}

We conclude this paper by surveying some directions for future research.

Given $\pi\in\mathfrak{P}_{n}$, we say that:
\begin{itemize}
\item $i\in[n-1]$ is a\textit{ left peak} of $\pi$ if $i$ is a peak of $\pi$, or if $i=1$ and $\pi_{1}>\pi_{2}$. Then $\Lpk\pi$ is defined to be the set of left peaks of $\pi$ and $\lpk\pi$ the number of left peaks of $\pi$. 
\item $i\in\{2,3,\dots,n\}$ is a\textit{ right peak} of $\pi$ if $i$ is a peak of $\pi$, or if $i=n$ and $\pi_{n-1}<\pi_{n}$. Then $\Rpk\pi$ is defined to be the set of right peaks of $\pi$ and $\rpk\pi$ the number of right peaks of $\pi$.
\end{itemize}
All four of these statistics are known to be shuffle-compatible \cite{Gessel2018}, and they were conjectured by Grinberg (along with $\Pk$, $\pk$, $\Val$, and $\val$) to be $M$-binomial. While we have resolved Grinberg's conjecture for the $\Pk$ and $\pk$ statistics, it remains open for these six other statistics, so we repeat this conjecture for these remaining statistics here.
\begin{conjecture}[{Grinberg \cite[Question 107]{Grinberg2018}}]
The statistics $\Lpk$, $\lpk$, $\Rpk$, $\rpk$, $\Val$, and $\val$ are $M$-binomial.
\end{conjecture}

Just as peaks and valleys are related by complementation, left peaks and right peaks are related by reversal. Given $\pi=\pi_{1}\pi_{2}\cdots\pi_{n}\in\mathfrak{P}_{n}$, define its \textit{reverse} $\pi^{r}$ by $\pi^{r}\coloneqq\pi_{n}\cdots\pi_{2}\pi_{1}$. Then 
\[
\Lpk\pi=(n+1)-\Rpk\pi^{r}\coloneqq \{\,n+1-i:i\in\Rpk\pi^{r}\,\}\qquad\text{and}\qquad\lpk\pi=\rpk\pi^{r},
\]
so these pairs of statistics are \textit{$r$-equivalent} (see \cite[Section 3.2]{Gessel2018} for the definition). Reversal is also shuffle-compatibility-preserving, so we have the $\mathbb{Q}$-algebra isomorphisms $\mathcal{A}^{\Lpk}\cong\mathcal{A}^{\Rpk}$ and $\mathcal{A}^{\lpk}\cong\mathcal{A}^{\rpk}$. On the level of compositions, let us define $L^{r}\vDash n$ for $L\vDash n$ by 
\[
L^{r}\coloneqq\Comp([n-1]\backslash(n-\Des L)), \quad\text{where}\quad n-\Des L\coloneqq \{\,n-i:i\in\Des L\,\}.
\]
Also, consider the involution $\rho$ on $\QSym$ given by $\rho(F_{L})=F_{L^{r}}$, which like $\psi$ is a $\mathbb{Q}$-algebra automorphism \cite[Section 3.6]{Luoto2013}. It is easily verified that $\Comp\pi=L$ implies $\Comp\pi^{r}=L^{r}$, and that $\rho$ restricts to an isomorphism between $\mathcal{K}^{\Lpk}$ and $\mathcal{K}^{\Rpk}$ as well as between $\mathcal{K}^{\lpk}$ and $\mathcal{K}^{\rpk}$. Upon obtaining an analogue of Theorem \ref{t-F} for the $\Lpk$ and $\lpk$ statistics, we could then establish analogous results for $\Rpk$ and $\rpk$ using the involution $\rho$ similar to the proof of Theorem \ref{t-vals}. 

Concerning these symmetries, we also pose the following question.
\begin{question}
If two statistics $\st_{1}$ and $\st_{2}$ are $r$-equivalent or $c$-equivalent, does $\st_{1}$ being $M$-binomial imply that $\st_{2}$ is also $M$-binomial?
\end{question}

On the other hand, the statistics $\maj$, $(\des,\maj)$, and $(\val,\des)$ are not $M$-binomial \cite[Question 107]{Grinberg2018}. Are there ways that we can tell whether a descent statistic is $M$-binomial solely in terms of the combinatorics of the statistic?
\begin{problem}
Find necessary and/or sufficient conditions for a descent statistic to have the $M$-binomial property (preferably combinatorial conditions).
\end{problem}

Grinberg \cite[Section 6]{Grinberg2018a} considers four binary operations $\prec$, $\succeq$, \textarm{\belgthor}, and \textarm{\tvimadur} on $\QSym$ and defines a notion of ideal for each of these operations. We may then naturally ask whether the kernel of a descent statistic is one of these types of ideals, which turns out to have combinatorial significance. Notably, $\mathcal{K}^{\st}$ is both an $\prec$-ideal and an $\succeq$-ideal if and only if $\st$ is LR-shuffle-compatible, in which case $\mathcal{A}^{\st}$ is canonically a dendriform algebra quotient of $\QSym$. We restate the following question of Grinberg.

\begin{question}[{Grinberg \cite[Question 6.16]{Grinberg2018a}}]
Which descent statistics $\st$ have the property that $\mathcal{K}^{\st}$ is an $\prec$-ideal, $\succeq$-ideal, \textarm{\belgthor}-ideal, and/or \textarm{\tvimadur}-ideal?
\end{question}

\noindent Grinberg has some results in this direction \cite[Section 6.7]{Grinberg2018a}; for example, the kernels $\mathcal{K}^{\Des}$, $\mathcal{K}^{\des}$, $\mathcal{K}^{(\des,\maj)}$, and $\mathcal{K}^{\Epk}$ are ideals with respect to all four of these operations, whereas $\mathcal{K}^{\Lpk}$ is an ideal with respect to all except \textarm{\tvimadur} (it is a left \textarm{\tvimadur}-ideal but not a right \textarm{\tvimadur}-ideal). Hence, the statistics $\Des$, $\des$, $(\des,\maj)$, $\Epk$, and $\Lpk$ are all LR-shuffle-compatible.

Finally, we note that the notions of shuffle-compatibility and shuffle algebras have recently been extended to cyclic permutations \cite{Domagalski[20202021],Liang}. A \textit{cyclic permutation} is an equivalence class of a (linear) permutation under cyclic rotation; for example, if $\pi=21948$ then
\[
[\pi]=\{21948,82194,48219,94821,19482\}.
\]
The cyclic shuffle algebra of any cyclic shuffle-compatible descent statistic is isomorphic to a quotient of $\cQSym^{-}$ \cite[Theorem 3.5]{Liang}, the ``non-Escher subalgebra'' of the $\mathbb{Q}$-algebra $\cQSym$ of \textit{cyclic quasisymmetric functions} introduced recently by Adin, Gessel, Reiner, and Roichman~\cite{Adin2021}. The \textit{fundamental cyclic quasisymmetric functions} $F_{[L]}^{\cyc}$, indexed by ``non-Escher cyclic compositions'', form a basis of $\cQSym^{-}$, and one can then define the kernel of a cyclic descent statistic $\cst$ to be the subspace of $\cQSym^{-}$ spanned by all elements of the form $F_{[J]}^{\cyc}-F_{[K]}^{\cyc}$ where $[J]$ and $[K]$ are $\cst$-equivalent cyclic compositions. In analogy with the linear setting, the kernel of $\cst$ is an ideal of $\cQSym^{-}$ if and only if $\cst$ is cyclic shuffle-compatible.
\begin{problem}
Study kernels of cyclic descent statistics.
\end{problem}

In particular, Theorems 2.8 and 3.7 of \cite{Liang} allow one (under certain conditions) to construct the cyclic shuffle algebra of a cyclic statistic from the shuffle algebra of a related shuffle-compatible (linear) statistic; this suggests an analogous result relating kernels of cyclic descent statistics and kernels of linear descent statistics. There is also an analogue of the monomial basis in $\cQSym^{-}$, so one can investigate the $M$-binomial property in this setting.

\bigskip{}

\noindent \textbf{Acknowledgements.} The work in this paper was initiated as part of the first author's honors thesis project at Davidson College (under the supervision of the second author); we thank Heather Smith Blake, Frankie Chan, Jonad Pulaj, and Carl Yerger for serving on the first author's thesis committee and giving feedback on his work. We also thank Darij Grinberg for helpful e-mail correspondence. Lastly, we thank an anonymous referee who very generously provided simpler proofs for Theorems \ref{t-span} and \ref{t-li}, as well as a number of minor corrections and suggestions which improved the presentation of this paper.

\bibliographystyle{plain}
\addcontentsline{toc}{section}{\refname}\bibliography{bibliography}

\end{document}